\documentclass[11pt]{amsart}
\usepackage[utf8]{inputenc}
\usepackage{ragged2e}
\usepackage{mathtools}
\usepackage{tikz}
\usepackage{amsfonts} 
\usepackage{amsmath}
\usepackage{enumerate}
\usepackage{amsthm}
\usepackage{amssymb}
\usepackage{dutchcal}
\usepackage{wasysym}
\usepackage{url}
\usepackage[foot]{amsaddr}
\usepackage{bbm}
\usepackage[left=2.5cm,right=2.5cm,top=3.2cm,bottom=3cm]{geometry}
\usepackage{cleveref}
\crefformat{section}{\S#2#1#3}
\crefformat{subsection}{\S#2#1#3}

\usepackage{pdfpages}
\usepackage{natbib}

\usepackage[nottoc]{tocbibind}

\makeatletter
\renewcommand{\fnum@figure}{Figure \thefigure}

\newtheorem{teo}{Theorem}
\newtheorem{lem}{Lemma}[section]
\newtheorem{prop}[lem]{Proposition}
\newtheorem{cor}{Corollary}[lem]
\newtheorem{defi}[lem]{Definition}

\theoremstyle{remark}
\newtheorem{rem}{Remark}[teo]
\newtheorem{reml}[cor]{Remark}

\newtheorem{q}{Question}

\def\C{\mathbb{C}}
\def\c{\mathcal{c}}
\def\e{\mathcal{E}}
\def\R{\mathbb{R}}
\def\Q{\mathbb{Q}}
\def\Z{\mathbb{Z}}

\def\A{\mathbb{A}}
\def\P{\mathbb{P}}

\def\F{\mathbb{F}}
\def\h{\mathbb{H}}
\def\H{\mathcal{H}}

\def\U{\mathcal{U}}
\def\I{\mathcal{I}}

\def\Ni{\text{Ni}}

\def\PGL{\text{PGL}}
\def\PSL{\text{PSL}}
\def\Gal{\text{Gal}}
\def\gq{G_{\Q}}
\def\gk{G_{K}}
\def\Cl{\text{Cl}}

\def\Aut{\text{Aut}}

\title{The arithmetic of critical values I: equicritical\\quartic polynomials}
\author{Francesco Naccarato}
\address{Department of Mathematics, ETH Zürich, Zürich CH-8092, Switzerland.}
\email{francesco.naccarato@math.ethz.ch}
\begin{document}

\begin{abstract}
    A polynomial $f$ of degree $d$ and coefficients in an algebraically closed field $k$ defines a morphism $f\colon\P^1_k\longrightarrow\P^1_k$ which, if char$(k)\nmid d$, is unramified outside a finite set of points in the image: the critical values of $f$. In this work we establish a rigorous framework for the study of their arithmetic, which we carry out for $d=4$ and $k=\overline{\Q}$, uncovering a connection to the arithmetic of elliptic curves. Recent progress in the theory of Weyl sums has sparked some interest in finding pairs of polynomials having the same critical values for ``nontrivial" reasons: building on our analysis, we provide a complete classification of such pairs in the case of quartics over number fields.
\end{abstract}

\maketitle

\section{Introduction}\label{1}

Let $k$ be an algebraically closed field, $d\ge2$ an integer such that char$(k)$ does not divide $d$, and let $f\in k[x]$ be a polynomial of degree $d$. The (finite) \textit{critical values} of $f$, denoted by $C_f$, are those $y\in k$ such that the equation $f(x)=y$ has less than $d$ distinct solutions---that is, the branching points of the ramified cover $f\colon\A^1_k\longrightarrow\A^1_k$ induced by $f$. They are obtained by evaluating $f$ at its \textit{critical points} $Z_f$, which are the roots of the derivative $f'$. These notions account for multiplicity, so $Z_f$ and $C_f$ are in principle multisets, and they are clearly finite. By compactifying, we can extend the definition of critical values to rational functions and, more generally, to ramified covers of $\P^1=\P^1_{k}$ (to which we will refer just as ``covers"):
\begin{defi}
    Let $X/k$ be a projective curve and let $\phi\colon X\longrightarrow\P^1$ be a morphism of degree $d$. The critical values of $\phi$ are its branching points, denoted by $\text{Branch}(\phi)$. Two morphisms $\phi_1,\phi_2\colon X\longrightarrow\P^1$ with the same critical values are called \normalfont{equicritical.}
\end{defi}
\noindent
Notice how this makes $\infty\in\P^1$ a critical value of all polynomials. 
There is a natural action of $\Aut_{k}(X)$ on the set of covers $\phi\colon X\longrightarrow\P^1$, given by $\gamma\cdot \phi = \phi\circ\gamma^{-1}$; critical values behave well under this action as we clearly have $\text{Branch}(\gamma\cdot \phi)=\text{Branch}(\phi)$. Of special interest for us is the case $X=\P^1$, with $\PGL_2(k)$ acting via fractional linear transformations. Observe that the subset of covers consisting of polynomials, which will be our focus, is preserved under this action precisely by the subgroup Aff$(k)$ of affine transformations. Following Arnold \cite{a}, we define:
\begin{defi}
    Two polynomials $f,g\in k[x]$ are said to be \textbf{topologically equivalent} if there exists $\lambda\in\text{Aff}(k)$ such that $f=g\circ\lambda$.
\end{defi}
\noindent
We refer to the equivalence classes of this relation as \textit{topological types,} adopting the notation $[f]$ for the topological type of $f$; if, moreover, $f\in K[x]$ for a subfield $K\subset k$, we also set $[f]_{K}=\{f\circ\lambda, \ \lambda\in\text{Aff}(K)\}$, referring to any such set as a \textit{$K$-topological type}. So, polynomials of the same topological type have the same critical values. In the following, when we refer to two polynomials as (in)equivalent, we mean with respect to topological equivalence. We also call a polynomial $f$ \textit{normalized} if it is monic and $f(0)=0$ (the normalization with $f'$ monic is usually chosen in the literature, see e.g$.$ \cite{BCN}: our choice is motivated by purely expository reasons and does not affect the nature of the results).

The study of critical values turns out to be crucial in understanding various phenomena, ranging from holomorphic dynamics \cite{ldm} to the inverse Galois problem \cite{har} to, somewhat surprisingly, exponential sums: while investigating polynomial value sets over finite fields and rings, Kowalski and Soundararajan formulated \cite[\S 1.2]{ks} a Fourier-analogue to the Davenport pair problem (see e.g$.$ \cite[\S 2]{Fr2} for an overview of this classical problem), observing that two integral polynomials whose reductions modulo a prime $p$ are equicritical (and have all critical points in $\F_p$) give rise to a \textit{Fourier-Davenport pair} over $\Z/p^2\Z$. They then point out that there are inequivalent such pairs in any degree at least $4$. It is then natural to ask (see \cref{5.3} for the application) if there exist inequivalent polynomials with the same global critical values, and not just modulo some prime. 

We saw that critical values are naturally attached to a topological type, rather than just a single polynomial; indeed, we can speak of equicritical topological types in the same way. Observe that $\text{Aff}(k)$ acts on the set of topological types of polynomials over $k$ by \begin{align}
    [f]^{\lambda}=[\lambda(f)],\label{outa}
\end{align}with the resulting topological type having finite critical values $\lambda(C_f)$. This yields an action of $\text{Aff}(K)$ on the set of $K$-topological types which, when char$(k)=0$, is generically free (for complex quartics there is only one exception, see Corollary \ref{fixp}). So, in this case, varying $\lambda\in\text{Aff}(K)$ allows us to produce infinitely many pairs of distinct equicritical $K$-topological types starting from a single one. In light of these remarks, we give the following definition:
\begin{defi}
    Given a subfield $K\subset k$, we set:\begin{align*}
        EC_d(K)=\{([f],[g])\colon f,g\in K[x], \ \deg f=\deg g=d, \ [f]\neq[g] \text{ and } \ C_f=C_g\}/\sim,
    \end{align*} where $\sim$ is the above $\text{Aff}(K)$-action taken component-wise. A \textbf{basis} for $EC_d(K)$ is a collection of polynomial pairs $(f,g)$ such that the pairs $([f],[g])$ form a set of representatives for $EC_d(K)$; a \textbf{quasi-basis} is a collection that can be extended to a basis by adding finitely many pairs.\label{basis}
\end{defi}
\begin{reml}
    One can slightly modify the definition of $\text{EC}_d(K)$ by replacing topological types with $K$-topological types, but not much changes in characteristic $0$: see Remark \ref{rtt}.
\end{reml}
\noindent
One may then ask the following:
\begin{q}[Equicriticality problem]
    What can we say about $EC_d(K)$ for given $K$ and $d$? When is it empty, and when is it an infinite set?
\end{q}
\noindent
This work deals with the case $d=4$ over number fields: we give, for any number field $K$, an explicit parametrization of a basis for $\text{EC}_4(K)$, showing in particular that this set is infinite.

Given an ordered pair $\c=(f,g)$ of polynomials, we let $\overline{c}$ be the pair $(g,f)$ and, for $\lambda\in \text{Aff}(k)$, we let $\lambda(\c)$ be the pair $(\lambda(f),\lambda(g))$. Set $\rho=1+\sqrt3, \ C=-720\sqrt3-1248, \ R=362+209\sqrt3$ and, for an algebraic number $\kappa$ of degree $2$, let $\overline\kappa$ be the other root of its minimal polynomial. Moreover, let $\omega\in\overline{\Q}$ be a primitive third root of unity and, for a number field $K$, define $K'=K\setminus\{\omega,\omega^2\}$ and $K''=K'\setminus\{0,1,-2,-2\omega,-2\omega^2,\rho,\overline{\rho},\omega\rho,\omega^2\rho,\omega\overline{\rho},\omega^2\overline{\rho}\}$. We can now state our main result: 
\begin{teo}\label{mainthm}
    Let $K$ be a number field. The collection $\e_K=(\c_t=(f_t,g_t), \ t\in K'\cup\{\infty\})$ given by:   
    \begin{align}\label{generic param}
    \c_t=\biggl(x^4-6tx^2-8x, \ -\frac{(t-1)^2}{3}\left(x^4-6\frac{t+2}{t-1}x^2-8x\right)-8(t^2+t+1)\biggl)\text{ for }t\in K'',\end{align}
    \begin{align*}\begin{split}
    \c_0&=\left(x^4-x, \ -\frac{1}{48}x^4-\frac{1}{4}x^2+\frac{1}{6}x-\frac{1}{2}\right)=\overline{c_{-2}},\\ 
    &\c_1=\biggl(x^4+2x^3, \ -\frac13x^4-2x^3-3x^2\biggl)=\overline{c_{\infty}}
    \end{split}
        \end{align*}
    and
    \vspace{-3mm}
    
    \begin{itemize}
        \item if $\sqrt3\in K$, \begin{align*}
            c_{\rho}=(f_{\rho}, \ -f_{\rho}+2C), \ c_{\overline{\rho}}=(f_{\overline{\rho}}, -f_{\overline{\rho}}+2\overline{C}),
        \end{align*}where $f_{\rho}(x)=x^4-6\rho^3x^2-8\rho^3x$ and $f_{\overline{\rho}}(x)=x^4-6\overline{\rho}^3x^2-8\overline{\rho}^3x$;
        \vspace{1mm}
        
        \item if $\omega\in K$, \begin{align*}
            \c_{-2\omega}=(g_0,\omega g_0)=\overline{\c_{-2\omega^2}};
        \end{align*}
        \item if $\sqrt3, \ \omega\in K$, \begin{align*}
            \c_{\omega\rho}=(f_{\rho}, \ iR(f_{\overline{\rho}}-\overline{C}))=\overline{\c_{\omega\overline{\rho}}}, \ \c_{\omega^2\rho}=(f_{\rho}, \ iR(-f_{\overline{\rho}}+\overline{C}))=\overline{\c_{\omega^2\overline{\rho}}},
        \end{align*}
    \end{itemize}
    is a basis for $\text{EC}_4(K)$. Moreover, quadruples of equicritical, pairwise inequivalent quartics over $K$ exist if and only if $\omega\in K$, in which case \eqref{generic param} produces the set of ordered pairs from such quadruples.
\end{teo}

\begin{rem}
    While it will follow from our proof, we can immediately check that the polynomials in a pair appearing in Theorem \ref{mainthm} are inequivalent: indeed, since (for $t\neq1,\infty$) the coefficients of $x^3$ are all $0$, a linear change of variable relating $f_t$ and $g_t$ would have to be a homothety, which is easily ruled out by looking at the other coefficients.
\end{rem}

Kuhn \cite[Satz 2]{ku} (then generalized by Kristiansen \cite[Theorem 1]{kri}) has shown that, for each ordered tuple $(y_1,...,y_{d-1})$ of real numbers such that $(-1)^{j+1}y_j+(-1)^jy_{j+1}>0$ for $j\in\{1,...,d-2\}$, there exists exactly one $\R$-topological type $[f]$ with real critical points $Z_f=\{x_1<...<x_{d-1}\}$ such that $f(x_j)=y_j$. Observe how, if we drop the alternating sign condition on the differences $y_{j+1}-y_j$, choosing $K=\Q$ in Theorem \ref{mainthm} gives an infinite family of counterexamples to the uniqueness part: indeed, one easily verifies that for $t>1$, $f_t$ and $g_t$ have real critical points. In other words, two inequivalent rational quartics with real critical points can share the same critical value set, but not the same critical ordered triple $(f(x_1),f(x_2),f(x_3))$, once the critical points $x_1<x_2<x_3$ are labeled in increasing order. Concerning the existence part, it is key that for $\R$ the absolute Galois group $\mathrm{Gal}(\C/\R)\simeq\Z/2\Z$ is small, so there is no hope for such a result to hold over number fields. Still, a natural question one can ask is:
\begin{q}
    Given a number field $K$, classify the critical value sets of degree $d$ polynomials over $K$.
\end{q} Call a three-element subset $S\subset\overline\Q$ \textit{critical over $K$} if there exists a quartic $f\in K[x]$ with $C_f=S$. The following result settles this question for $d=4$:
\begin{teo}
    Let $K$ be a number field and let $y_1,y_2,y_3$ be distinct algebraic numbers such that the elliptic curve $E:z^2=(y-y_1)(y-y_2)(y-y_3)$ satisfies $j(E)\neq1728$. The set $\{y_1,y_2,y_3\}$ is critical over $K$ if and only if the above model for $E$ is defined over $K$ and $$j(E)=\frac{(u+3)^3(u+27)}{u}$$ for some $u\in K^*$. If $j(E)=1728,$ $\{y_1,y_2,y_3\}$ is critical over $K$ if and only if $E$ is a $K$-quadratic twist of $E_{\rho}:z^2=y^3+(1+2\rho)y$ or of $E_{\bar\rho}:z^2=y^3+(1+2\bar\rho)y$.\label{th2}
\end{teo}

\begin{rem}
    The statement of Theorem \ref{th2} suggests that it is useful to interpret the critical values of a quartic as the $2$-torsion of an elliptic curve: we will see why this is the case in \cref{3.1}. The theorem actually provides a characterization of the curves that arise this way; still, many questions about these curves and their connection to critical values remain worth asking: for instance, one may be interested in classifying the integral $j$-invariants that are realized for $K=\Q$, in which case an elementary $p$-adic argument shows that only $u=\pm 3^k, \ 0\le k\le6$, work. Additional questions and applications, especially in light of the constructions of \cref{4.2}, will be the subject of a separate upcoming work of the author.
\end{rem}

The paper is structured as follows: in \cref{2} we review the literature on critical values of polynomials, setting up the algebro-geometric picture of the problem; in doing so, we prove a generalization of a result of Arnold, which may bear some independent interest. A key ingredient in the proof of Theorem \ref{mainthm} is a thorough study of the (reduced) Hurwitz space $\H_4$ associated to quartic covers of $\P^1$ with monodromy of polynomial type: in \cref{3}, after introducing Hurwitz spaces, we construct the aforementioned reduced space for any degree $d\ge3$, and we study its geometry and arithmetic. \cref{4} is centered around $\H_4$: we prove that it is a certain well-known modular curve, and use this to establish Theorem \ref{th2}. \cref{5} is devoted to the equicriticality problem for quartics over number fields: we first give a nonconstructive proof---which may adapt to similar questions---for the existence of a parametrization like that of Theorem \ref{mainthm}, and then go on to compute the explicit equations for the polynomials appearing in the statement, by exploiting the rich structure arising from the link to the arithmetic of elliptic curves.

\subsection{Acknowledgments}

I am grateful to my advisor Emmanuel Kowalski, for recommending this topic and for the helpful advice he offered during the development of this work; to Davide Lombardo, for working out the details of the proof of Lemma \ref{qiso} and for his suggestion of following an explicit approach to the construction of the Hurwitz space; to Jordan Ellenberg, for sketching the ideas behind the fiber product formulation of the equicriticality problem and for pointing me to the paper of Rubin and Silverberg on elliptic curves with constant mod $p$ representations. Finally, I thank the anonymous referees for their many helpful suggestions.

\noindent
The author was partially supported by the Swiss National Science Foundation (grant number 219220). This work is part of ongoing research within the scope of the author's PhD thesis. 

\section{Critical values of polynomials: arithmetic and geometry}\label{2}
From now on, we assume that either char$(k)=0$ or char$(k)>d$. Given a normalized polynomial $f\in k[x]$, we attach to it the multiset of its critical points, as a (closed) point $\{x_i\in Z_f\}\in\A^{d-1}/S_{d-1}$. Since with any point $\widehat{x}\in\A^{d-1}/S_{d-1}$ we can associate the normalized formal integral (always assumed to have $0$ constant term) of $\prod_{\widehat{x}=\{x_i\}}(x-x_i)$ multiplied by $d$, we get a bijection:
\begin{align}\begin{split}
\{\text{normalized polynomials } f\in k[x]\text{ of degr}&\text{ee} \ d\}
\overset{1:1}{\longleftrightarrow}
\A^{d-1}/S_{d-1}\\ 
 &f \ \ \ \ \ \longmapsto \ \ \ \ \ \ Z_f.    
\end{split}\label{bij}
\end{align}
\noindent
If one does not care about the subfield of definition of the coefficients, the study of polynomials with given critical values can be carried out just for normalized polynomials: the leading coefficient of $f(ax+b)$ is $a^d$ times that of $f$, and its constant term is $f(b)$. Therefore, in light of \eqref{bij}, the starting point for our analysis will be the relation between critical points and critical values. We will then recover information about fields of definition via the study of rational points on Hurwitz spaces. 

If the normalized polynomial $f$ corresponds to $\widehat x=\{x_1,...,x_{d-1}\}\in\A^{d-1}/S_{d-1}$, its finite critical values $y_1,...,y_{d-1}$ are given by:
\begin{align}
    y_j=\sum_{i=0}^{d-1}(-1)^i\frac{d}{d-i}e_i(\widehat x)x_j^{d-i}=:\theta_d(x_1,...,x_{d-1})_j,\label{var}
\end{align}
where $e_i(\widehat x)$ is the elementary symmetric function of degree $i$ in $d-1$ variables (and $e_0\equiv1$). This translates the study of critical values for polynomials of degree $d$ to that of the covering $\theta_d\colon\A^{d-1}\longrightarrow\A^{d-1}$ defined by \eqref{var}, whose fiber above a point $(y_1,...,y_{d-1})$ of $\A^{d-1}$ corresponds under \eqref{bij} (and the quotient) to the set of normalized polynomials of degree $d$ having as finite critical values its coordinates. An elementary but important observation to make is that, for any subfield $K\subset k,$ the finite critical values of $f\in K[x]$, taken with their multiplicity, are the (only) roots of a polynomial defined over $K$: this can be seen directly by letting $\gk=\text{Gal}(\overline{K}/K)$ act on both sides of \eqref{var}, since $\theta_d$ is defined over $K$. 

Let $U_d$ be the affine open set $\{\prod_{1\le i < j\le d-1}(y_i-y_j)\neq 0\}\subset\A^{d-1}$ and let $V_d\subset U_d$ be the closed subscheme given by $y_1=0$. Polynomials with finite critical values in $U_d$---that is, pairwise distinct---are called \textit{Morse polynomials}. Arnold \cite[p.5, Theorem]{a}, following Davis \cite{d}, Thom \cite{t} and Zdravkovska \cite{z}, shows that for $d\ge3$ there are $d^{d-3}$ topological types of complex Morse polynomials of degree $d$ with given critical values. We generalize his result to our fields $k$ which, we recall, are algebraically closed and have characteristic either $0$ or larger than $d$:
\begin{prop}\label{dd3}
     Let $d\ge3$. For any $(y_1,...,y_{d-1})\in U_d$ we have $\#\{[f]\colon f\in k[x], \ C_f=\{y_1,...,y_{d-1}\}\}=d^{d-3}$.
\end{prop}
\begin{reml}
    Proposition \ref{dd3} shows that, in characteristic different from $2$ and $3$, the equicriticality problem is trivial for cubic polynomials: we have $EC_3(k)=\emptyset$.
\end{reml}
\noindent
Arnold's proof reduces the question to that of counting trees on $d$ vertices, which had already been solved by Cayley. Our strategy builds on Beardon, Carne and Ng's study \cite[Lemma 2.4]{BCN}, over $\C$, of the Jacobian $J_d$ of $\theta_d$ above nonzero distinct critical values, extending it to our fields $k$ and to the case where one of the critical values is $0$ (for a generalization of their study in the case of equal but nonzero critical values, see \cite[Theorem B]{DM}). Let the morphism $\widetilde\theta_d$ be given by \eqref{var} except for the equation for $y_1$, where the right-hand side is divided by $x_1$.
\begin{lem}
\begin{enumerate}[(i)]
    \item The morphism $\theta_d$ restricted to $\theta_d^{-1}(U_d\setminus V_d)$ is étale;
    \item the morphism $\widetilde\theta_d$\ restricted to $\widetilde\theta_d^{-1}(V_d)$ is étale.
\end{enumerate}\label{jac}
\end{lem}
\begin{proof}
\begin{enumerate}[(i)]
    \item Let $Q\in U_d\setminus V_d$, $P\in\theta_d^{-1}(Q)$ and set $J=J_d$ for convenience. Let $I_d\subset k[x_1,...,x_{d-1}]$ be the ideal generated by the polynomials defining $\theta_d(x_1,...,x_{d-1})=Q$, and set $X_d=\text{Spec}(k[x_1,...,x_{d-1}]/I_d)$. The part of the proof of Lemma 2.4 in \cite{BCN} where the authors show invertibility of $J$ at $P$ carries over to the fields we are considering, since it only relies on integral identities that hold formally in this case and on a polynomial of degree $\le d$ with $d+1$ roots being identically $0$. Therefore, the cotangent space of $X_d$ at $P$ has dimension $(d-1)-\text{rk}(J(P))=0$, so in particular it must be equal to the local dimension of $X_d$ at $P$, and the Jacobian criterion (which we can apply even in positive characteristic since $k=\overline{k}$) tells us that $\theta_d$ is smooth at $P.$  Clearly the local dimension at $P$ is an upper bound for the relative dimension at $P$ and hence, since smooth morphisms of relative dimension $0$ are étale, we get the desired claim by arbitrariness of $P$ and $Q$.
    \item Let now $Q=(y_1,...,y_{d-1})\in V_d$, $P=(x_1,...,x_{d-1})\in\widetilde\theta_d^{-1}(Q)$ and set $\widetilde J=\text{Jac}_{\widetilde\theta_d}$. We want to show independence of the columns of $\widetilde J(P),$ from which the claim will follow as in (i). Suppose on the contrary that there exist $\lambda_1,...,\lambda_{d-1}\in k$ such that $\sum_{j=1}^{d-1}\lambda_j\widetilde J_{ij}(P)=0$ for $1\le i\le d-1$. Let us start by assuming $x_1\neq0$; then, respectively from the first equation in the proof of Lemma 2.4 in \cite{BCN} and from the definition of $\widetilde\theta_d$ (our different normalization for $\theta_d$ by a factor of $d$ does not impact these equations), we get: \begin{align}
        \int_0^{x_i}\sum_{j=1}^{d-1}\lambda_j\prod_{\substack{1\le k\le d-1\\ k\neq j}}(w-x_k) \ dw=0, \ \ \ \ \ \ 2\le i\le d-1\label{int}
    \end{align}and\begin{align}
        x_1\int_{0}^{x_1}\sum_{j=1}^{d-1}\lambda_j\prod_{\substack{1\le k\le d-1\\ k\neq j}}(w-x_k) \ dw - \lambda_1\int_{0}^{x_1}\prod_{1\le k\le d-1}(w-x_k) \ dw=0.\label{int2}
    \end{align}
    Now, observe that the second integral in \eqref{int2} multiplied by $d$ is just the value at $x_1$ of the normalized polynomial with critical points $x_1,...,x_{d-1}$, so none other than $y_1=0$. Therefore, denoting by $f_{\lambda}(w)$ the common integrand left in both equations, we see that the polynomial $F_{\lambda}(x)=\int_0^xf_{\lambda}(w) dw$ has degree at most $d-1$ and vanishes for $x=0,x_1,...,x_{d-1}$, and hence vanishes identically. But then so does its derivative $f_{\lambda}$; on the other hand, for any $1\le j\le d-1$, we find $f_{\lambda}(x_j)=\lambda_j\prod_{\substack{1\le k\le d-1\\ k\neq j}}(x_j-x_k)$: since $P\in U_d,$ this means $\lambda_j=0,$ and we are done.
    \noindent
    
    If $x_1=0$, \eqref{int} still holds, so $F_{\lambda}(x)=\int_0^xf_{\lambda}(w) dw$ vanishes at $0,x_2,...,x_{d-1}$; therefore, if it does not vanish identically---in which case we conclude as above---we must have \begin{align}
        F_{\lambda}(x)=Cx\prod_{2\le j\le d-1}(x-x_j)\label{nv}
    \end{align} for some $C\in k^*$. On the other hand, a simple direct computation with \eqref{var} yields $\widetilde J(P)_{(1,1)}=(-1)^{d-1}d\frac{\prod_{2\le j\le d-1}x_j}{2}\neq0,$ while $\widetilde J(P)_{(1,j)}=0$ for $2\le j\le d-1$ as $x_1$ appears in each monomial of these partial derivatives. This immediately gives $\lambda_1=0$. But then, we find:\begin{align*}
        0\overset{\eqref{int}}{=}f_{\lambda}(x_1)\overset{x_1=0}{=}F'_{\lambda}(0)\overset{\eqref{nv}}{=}(-1)^{d-2}C\prod_{2\le j\le d-1}x_j\neq0,
    \end{align*}
    which yields the desired contradiction.
\end{enumerate}
\end{proof}
\noindent
We are now ready to prove Proposition \ref{dd3}:
\begin{proof}[Proof of Proposition \ref{dd3}]
As $\theta_d$ is étale over $U_d\setminus V_d$ by Lemma \ref{jac} (i), it is unramified with finite fibers. Applying Bezout's Theorem, we find that the cardinality of the fiber above $Q\in U_d\setminus V_d$ equals $d^{d-1}$ minus the number of points at infinity of \eqref{var}; as its right-hand side is homogeneous, the points at infinity are found by letting it vanish for $1\le j\le d-1$, and hence correspond to the \textit{nonzero} vectors of critical points of normalized polynomials having all finite critical values $0$. But a polynomial $h$ with $C_h=\{0\}$ (as a multiset) must vanish at all roots of $h'$, necessarily with higher multiplicity, which implies that $h'$---and hence $h$---have a single root, of multiplicity equal to the degree. Therefore, if $h$ is normalized then we must have $h(x)=x^d$, showing that there are no intersection points at infinity: we have $|\theta_d^{-1}(Q)|=d^{d-1}$. Now, any normalized polynomial $f$ in our fiber is equivalent to $d^2$ normalized polynomials in the same fiber, given by $f(\mu_d x+b)$ for $\mu_d$ any $d$-th root of unity in $k$ and $b$ any of the $d$ distinct roots of $f$. Therefore, we are left with $d^{d-3}$ topological types over each point of $U_d\setminus V_d$. 

\noindent
Let now $Q\in V_d$, and assume that the entry of $Q$ that is $0$ is the first one (we can safely do so, since all our equations are symmetric). If $f\in\theta_d^{-1}(Q),$ then one of the critical values of $f$ is $0$, so $f$ has a common root with $f'$ and hence a double root---exactly one, since the critical values are distinct. Therefore, in computing the number of topological types in each fiber, we need to divide its cardinality by $d(d-1)$ instead of $d^2$. On the other hand, in virtue of how $\widetilde\theta_d$ is defined, the preimage of $Q$ under $\theta_d$ has the same number of elements (without multiplicity) as the preimage of $Q$ under $\widetilde\theta_d$, since the first equation of \eqref{var} has the form $f_1(x_1,...,x_{d-1})=0$ with $x_1^2\mid f_1$. But Lemma \ref{jac} (ii) tells us that $\widetilde\theta_d$ is étale over $V_d$, so by the same argument as above it has fibers of cardinality $(d-1)d^{d-2},$ and we are done.
\end{proof}

\begin{reml}
    As kindly pointed out by a referee, a quicker way of deducing Proposition \ref{dd3} is from the fact that the underlying Hurwitz space and its map to the configuration space can be defined over $\Z[\frac{1}{d!}$], so the degree of this map over any field of characteristic larger than $d$ is the same as its degree over $\C$.
\end{reml}

\section{The Hurwitz space of Morse polynomials}\label{3}

Hurwitz spaces are, roughly speaking, moduli spaces of covers $X\longrightarrow\P^1$ with given ramification properties. We assume for simplicity that the field of definition $k$ of our covers is $\C$ (the following constructions generalize to any algebraically closed field of characteristic $0$, but the nice topological setting offered by Riemann surfaces is missing; see \cite[\S 1.1]{Ca} for an account), keeping in mind that compact Riemann surfaces come from algebraic curves by Serre's GAGA. So, let $\P^1=\P^1_{\C}$ and consider all objects as base-changed to $\C$ when necessary. Moreover, let $\U_r$ be the moduli space for $r$ distinct unordered points in $\P^1$ equipped with its natural algebraic structure of affine variety over $\Q$, whereby the field of definition of the point $(y_1,...,y_r)$ is that of the polynomial $\prod_{y_i\neq\infty}(y-y_i)$.

\subsection{Moduli spaces of $G$-covers}\label{3.1}

We begin by giving an outline of the constructions in Hurwitz theory that are relevant for our study; a more comprehensive introduction can be found, e.g., in the work of Fried and Völklein \cite{FV}. Let $r$ and $d$ be positive integers, let $G$ be a subgroup of $S_d$ and let $C=(C_1,...,C_r)$ be an unordered tuple of conjugacy classes in $G$. We will refer to such a triple $(r,G,C)$ as a \textit{datum,} where the embedding of $G$ in $S_d$ is implicit.
\begin{defi}\label{ncl}
    A \textbf{Nielsen class} for the triple $(r,G,C)$ is an element of \begin{align*}
        \Ni_r(G,C)=\{(g_1,...,g_r)\in G^r\colon \exists\sigma\in S_r\colon g_i\in C_{\sigma(i)} \ \forall i=1,...,r, \ g_1\cdot...\cdot g_r=1 \text{ and }\langle\{g_i\}\rangle=G\}.
    \end{align*} An \textbf{inner Nielsen class} is an element of $Ni_r(G,C)^{\text{in}}=Ni(G,C)/G$, where $G$ acts component-wise by conjugation.\\
    An \textbf{absolute Nielsen class} is an element of $Ni_r(G,C)^{\text{abs}}=Ni(G,C)/N_{S_d}(G)$, where the normalizer acts component-wise by conjugation.
\end{defi}

Recall that a $G$-cover is a pair $(\phi\colon X\longrightarrow\P^1,\delta)$---usually denoted just by $\phi$---where $\phi$ is a Galois cover and $\delta\colon\Aut(\phi)\longrightarrow G$ is a group isomorphism. Let $\phi$ have $r_{\phi}=r$ critical values $y_1,...,y_r$, let $y_0\in\P^1$ be any other point and, for $1\le i\le r$, let $U_i$ be a simply-connected open neighborhood of $y_i$ that does not contain any other critical value but contains $y_0$. Then, taking a generator $\rho_i$ of $\pi_1(U_i\setminus\{y_i\},y_0)$ with positive orientation and lifting it to $X$ gives a deck transformation $g_i\in G$ of $X$, via the isomorphism $\delta$. Let $\Cl(g_i)$ be the respective conjugacy class in $G$ and denote by $C(\phi)$ the unordered tuple $(\Cl(g_1),...,\Cl(g_r))$: we say that $\phi$ branches with local monodromy $C(\phi)$; this is a good definition since picking different, positively-oriented generators for the fundamental groups only changes the induced deck transformations by inner automorphisms of $\Aut(\phi)$. We have just described a way to attach to a $G$-cover $\phi$ a datum, namely $(r_{\phi},\Aut(\phi),C(\phi)),$ and by choosing the $\rho_i$'s so that $\prod_{i=1}^r\rho_i=1$ under the inclusions $\pi_1(U_i\setminus\{y_i\},y_0)\hookrightarrow\pi_1(\P^1\setminus\{y_1,...,y_r\},y_0), \ 1\le i\le r$, we have also attached to $\phi$ the Nielsen class $(g_1,...,g_r).$

As in the case of polynomials, it is more natural to attach a datum to a suitable equivalence class of $G$-covers: we generalize the definition of topological equivalence by saying that two $G$-covers $(\phi:X\longrightarrow\P^1, \ \delta)$ and $(\psi:Y\longrightarrow\P^1, \ \xi)$ are equivalent, or isomorphic, if there is an isomorphism of covering spaces $\gamma:X\longrightarrow Y$ that commutes with the respective actions of $G$, i.e. $\delta(g)=\xi(\gamma g\gamma^{-1}) \ \forall g\in\Aut(\phi)$. We adopt the terminology \textit{$G$-cover class} and the notation $(\phi)$ for the equivalence class of $\phi$ under this relation. Isomorphic covers clearly have the same critical values, while the additional condition on the $G$-actions implies that the datum associated to a $G$-cover is an invariant of its class. Hence, we have a (noncanonical, due to its dependence on the choice of generators $\rho_i$) map
\begin{align}
\Phi:
\left\lbrace
\begin{aligned}
\displaystyle
\text{classes of $G$-covers with datum } \\ \text{$(r,G,C)$ branching at $y_1,...,y_r$}
\end{aligned}
\right\rbrace
\longrightarrow \ \ 
\text{Ni}_r(G,C)^{\text{in}},   \label{nc}
\end{align}
\noindent
sending $(\phi)$ to the inner Nielsen class representative of $(g_1,...,g_r)$ as defined above. The following result (see \cite[\S 1.2]{Fr1} and \cite[\S 2.2.1]{Fr2}) shows that $\Phi$ is a bijection: 
\begin{teo}[Riemann Existence Theorem, inner version]\label{ret}
    Given a datum $(r,G,C)$ and $r$ points $y_1,...,y_r$ in $\P^1$, $Ni(G,C)^{\text{in}}$ is in bijection with the set of classes of $G$-covers $\phi:X\longrightarrow\P^1$ that branch over $y_1,...,y_r$ with local monodromy $C$.
\end{teo}
\begin{rem}
    Here the genus of $X$ is uniquely determined by the datum: this is the Branch Cycle Lemma, see \cite[\S 5.1]{Fr3}.\label{bcl}
\end{rem}
Notice how in Theorem \ref{ret} the branching set is fixed, while in the situation that interests us it varies. The Riemann Existence Theorem for families (see again \cite[\S 1]{Fr1}) enables us to construct Hurwitz spaces as parameter families of Nielsen classes over $\U_r$: in particular (see \cite[Theorem 1.1]{Ca}),
the \textbf{inner Hurwitz space} \begin{align}\label{inhs}
    \Psi_r\colon H_r^{\text{in}}(G,C)\longrightarrow \U_r
\end{align}
has the property that for any $Q=(y_1,...,y_r)\in\U_r$, the fiber above $Q$ is in bijection with the set of $G$-cover classes that branch at $Q$ with local monodromy $C$. Moreover, it admits a \textit{unique} algebraic model as an affine variety over a number field $k_0=k_0(C)$ such that $\Psi_r$ is defined over $k_0$ and, for any algebraic extension $K/k_0$, any $\sigma\in\gk$ and any $(\phi)\in H_r^{\text{in}}(G,C)(\overline{\Q})$, we have: \begin{align}
    ^{\sigma}(\phi)=(^{\sigma}\phi).\label{galc}
\end{align}

As mentioned at the end of \cref{1}, our focus will mainly lie on a \textit{reduced} Hurwitz space. Observe that there is an \textit{outer} $\PGL_2(\C)$-action on $G$-cover classes, given by $(\phi)^{\alpha}=(\alpha(\phi))$. We say that two $G$-cover classes in the same orbit are \textit{$\PGL_2(\C)$-equivalent,} and refer to an equivalence class as a \textit{$G/\PGL_2$-cover class}. Then, given the same datum as above, the reduced Hurwitz space $H_r^{\text{red}}(G,C)$ is the affine variety of dimension $r-3$ obtained as the quotient of $H_r^{\text{in}}(G,C)$ by this action. Indeed, as we already observed in \eqref{outa} for polynomials and affine transformations, we have $\text{Branch}(\alpha(f))=\alpha(\text{Branch}(f)),$ and local monodromy around critical values is clearly preserved by the action. This gives a commutative diagram:
    
    \begin{center}
        \begin{tikzpicture}[>=stealth,->,node distance=3cm]
  \node (X) at (0,3) {$H_r^{\text{in}}(G,C)$};
  \node (Y) at (3,3) {$H_r^{\text{red}}(G,C)$};
  \node (P1) at (0,0) {$\U_r$};
  \node (P2) at (3,0) {$\I_r$};

  \draw[->] (X) to node[above] {$\Pi_r$} (Y);
  \draw[->] (X) to node[left] {$\Psi_r$} (P1);
  \draw[->] (Y) to node[right] {$\beta_r$} (P2);
  \draw[->] (P1) to node[above] {$\pi_r$} (P2);
\end{tikzpicture}
    \end{center}
where $\I_r=\U_r/\PGL_2(\C)$ and the maps $\Psi_r,\Pi_r$ and $\beta_r$ depend on the datum. A classical result of Mumford and Fogarty \cite[Theorem 1.1]{mf} shows that both quotient spaces are well-defined in the category of affine $k_0$-varieties. In virtue of its construction, $K$-rational points on $H_r^{\text{red}}(G,C)$ correspond with $\PGL_2(\C)$-orbits of $G$-cover classes $(\phi)$ (with $\phi$ taken to be defined over $\overline{K}$) such that for each $\sigma\in\gk$ there exist $\alpha_{\sigma}\in\PGL_2(\C)$ and an isomorphism $\gamma_{\sigma}$ of $G$-covers from $\alpha_{\sigma}(^{\sigma}\phi)$ to $\phi$.

If $r=4$ then $\I_4$ is the usual moduli space of elliptic curves $Y(1)$, since the $j$-invariant of an elliptic curve $E$ equals \begin{align}j(\lambda)=256\frac{(\lambda^2-\lambda+1)^3}{\lambda^2(\lambda-1)^2},\label{jquad}\end{align} with $\{0,1,\lambda,\infty\}$ a representative in the $\PGL_2(\C)$-orbit of the $x$-coordinates of $E[2]$. We will hence from now on refer to the \textit{$j$-invariant of a quadruple of distinct complex numbers} without this causing any confusion. A fiber of the morphism  \begin{align*}
    \beta_4\colon H_4^{\text{red}}(G,C)\longrightarrow Y(1)
\end{align*} is then naturally identified with the set of $G/\PGL_2$-cover classes with monodromy $C$ and representatives that branch over the $2$-torsions of the elliptic curves with $j$-invariant at the base. We will refer to the latter as the \textit{critical $j$-invariant} of a $G$- or $G/\PGL_2$-cover class.

By construction, $\beta_4$ only ramifies above critical $j$-invariants of quadruples with ``large" stabilizer in $\PGL_2(\C)$, i.e$.$ the elliptic $j$-invariants $0$ and $1728.$ It is then well known that $H_r^{\text{red}}(G,C)$ must be a quotient of the upper half-plane $\h$ by a finite index subgroup $\Gamma=\Gamma(G,C)$ of the modular group $\PSL_2(\Z)$, with $\beta_4$ induced by the inclusion $\Gamma<\PSL_2(\Z)$. In other words, we are dealing with a (not necessarily congruence) modular curve. See again \cite[\S 3]{Fr1} for a general overview of reduced Hurwitz spaces of dimension $1$ as modular curves.

\subsection{Lifting points on Hurwitz spaces, and the case of polynomials}

Given $x\in H_r^{\text{red}}(G,C)$, we call any of its preimages in $H_r^{\text{in}}(G,C)(K)$ a ``lift" of $x$. The question of which points $x\in H_r^{\text{red}}(G,C)(K)$ lift to $H_r^{\text{in}}(G,C)(K)$ is a rich and open one, whose main application is in the domain of the inverse Galois problem: indeed, the IGP reduces to finding rational points on inner Hurwitz spaces, thanks to Hilbert's Irreducibility Theorem. Looking for $K$-rational points on $H_r^{\text{red}}(G,C)$ is easier due to it being a rational quotient of lower dimension, especially in the case $r=4$ when it is a curve. However, these points rarely lift to points of $H_r^{\text{in}}(G,C)(K)$; this behavior is largely controlled by two factors:\begin{enumerate}[(a)]
    \item the lifting of points from $\I_r(K)$ to $\U_r(\overline{K})$;
    \item the \textit{base invariant} of the $G/\PGL_2(\C)$-cover class $x$, i.e. the conjugacy class in $\PGL_2(\C)$ of the stabilizer of any of its lifts to $H_r^{\text{in}}(G,C)$.
\end{enumerate} 
In \cite{Ca}, Cadoret gives a thorough account of the theory and proves some strong lifting results, especially for $G/\PGL_2$-cover classes with nontrivial base invariant: see for instance \cite[Lemma 3.12, (2)]{Ca}. In point (1) of the same result, the author observes that, for $r=4$, a point in $\I_r(K)$ always lifts to a point in $\U_r(K)$, as this is equivalent to fields of moduli being fields of definition for elliptic curves. Notice that a stronger condition than $x$ having trivial base invariant is the $\PGL_2(\C)$-stabilizer of the branching set of any (i.e$.$ all) of its lifts being trivial. For our purposes, it is enough to prove the following lifting result:
\begin{lem}
    For any datum $(r,G,C)$, any tower of algebraic extensions $K_0/K/k_0(C)$ and any $x\in H_r^{\text{red}}(G,C)(K)$, if:\vspace{-2mm} \begin{enumerate}[(1)] \item $\beta_r(x)$ lifts to a point in $\U_r(K_0)$;
        \item $\text{Stab}_{\PGL_2(\C)}(\text{Branch}(\phi_0))=\{I\}$ for some lift $(\phi_0)$ of $x$,
    \end{enumerate}\vspace{-2mm}
    then there exists a $G$-cover $\phi:X\longrightarrow\P^1$ such that $(\phi)\in H_r^{\text{in}}(G,C)(K_0)$ and $\Pi_r((\phi))=x$.\label{redlift}
\end{lem}
\begin{proof}
    Let $Q\in\U_r(K_0)$ be a lift of $\beta_r(x)$ and let $(\phi)$ be any $G$-cover class in $H_r^{\text{in}}(G,C)(\overline{\Q})$ such that $\Pi_r((\phi))=x$ and $\Psi_r((\phi))=Q$ (which exists as the $\PGL_2(\C)$-action on $G$-cover classes and that on critical values are compatible). We know that for each $\sigma\in\Gal(\overline{K_0}/K_0)$ there exist $\alpha_{\sigma}\in\PGL_2(\C)$ and an isomorphism of $G$-covers $\gamma_{\sigma}\in\Aut(X)$ from $\alpha_{\sigma}(^{\sigma}\phi)$ to $\phi$. Since $\text{Branch}(^{\sigma}\phi)= {^{\sigma}}Q=Q$, $\alpha_{\sigma}$ must lie in its stabilizer, which is conjugate to that of $\phi_0$ and hence trivial. Therefore, $\alpha_{\sigma}$ is the identity, implying that $^{\sigma}\phi=\phi(\gamma_{\sigma})$ and that $\gamma_{\sigma}$ commutes with the $G$-actions of $\phi$ and $^{\sigma}\phi$; in other words, $(\phi)\in H_r^{\text{in}}(G,C)(K_0)$ by \eqref{galc} and the paragraph preceding it.
\end{proof}

Let us now go back to polynomials. A Morse polynomial $f\in k[t]$ defines a cover $f\colon\P^1\longrightarrow\P^1$ of the same degree $d$, branching at $C_f\cup\{\infty\}$ with local monodromy $C_{\text{pol}}$ given by a transposition at each of the finite critical values and a $d-$cycle at $\infty$ (see \cite[\S 2.2.2]{Fr2}). It is a classical result going back to Hilbert that any choice of representatives for the classes in $C_{\text{pol}}$ whose product is the identity must generate the full symmetric group (see \cite[\S 4.4]{Serre1992}), so the Galois group of the cover $f(t)-y$ is $S_d$. 
In order to apply Hurwitz theory to our question, we set $r=d, \ G=S_d$ and $C=C_{\text{pol}}$. Then necessarily $X=\P^1$ (see Remark \ref{bcl}), and we get the inner Hurwitz space\begin{align*}
    \Psi^{\text{pol}}_d\colon H_d^{\text{in}}(S_d,C_{\text{pol}})\longrightarrow\U_d,
\end{align*}whose fiber above $\{y_1,...,y_d\}\in\U_d$ is in bijection with the set of classes of $S_d$-covers $\phi\colon\P^1\longrightarrow\P^1$ branching over the $y_i$'s with local monodromy $C_{\text{pol}}$. We call these covers \textit{polynomial $S_d$-covers.} Since $C_{\text{pol}}^m=C_{\text{pol}}$ for odd $m$, the results in \cite[\S 1.1]{FV} imply that our inner Hurwitz space is defined over $\Q$. 

Notice the similarity between $\Psi_d^{\text{pol}}$ and the map $\theta_d$ of Section 2: if we denote by $\iota_d$ the inclusion $U_d/S_{d-1}\hookrightarrow\U_d$ given by $\{y_1,...,y_{d-1}\}\mapsto\{y_1,...,y_{d-1},\infty\}$, and by $\mathcal{W}_d$ its image, then $\Psi_d^{\text{pol}}|_{\Psi_d^{-1}(\mathcal{W}_d)}$ and $\iota_d\circ\theta_d$ are both rational maps from an affine variety to $\mathcal{W}_d$ whose fibers convey information about the polynomials branching at the base. Nevertheless, it is not the \textit{same} information: $\Psi_d^{\text{pol}}$ accounts for all classes of covers with polynomial monodromy, not just (normalized) polynomials, and more importantly, it carries data about the field of definition of the covers by \eqref{galc} (see Theorem \ref{lifts}). While polynomials are not Galois covers themselves, we have:
\begin{lem}\label{prec}
    Polynomial $S_d$-cover classes with given branch locus $Q$ are in bijection with isomorphism classes of covers with branch locus $Q$ and local monodromy $C_{\text{pol}}$. Moreover, each such class is $\PGL_2(\C)$-equivalent to the isomorphism class of a polynomial.
\end{lem}
\begin{proof}
    A bijection analogous to \eqref{nc} holds for absolute Nielsen classes, in this case with isomorphism classes of (mere) covers with given datum and branch locus ($G$ being the automorphism group of the Galois closure): this gives rise \cite[Proposition 6.3]{Fr3} to the \textit{absolute Hurwitz space} relative to the datum. If $G=S_d$, inner and absolute Nielsen classes coincide in virtue of Definition \ref{ncl}, so in this case the inner and absolute Hurwitz spaces are isomorphic (see \cite[\S 3.1.3]{BF} for more details).

    Since $\PGL_2(\C)$ acts transitively on $\P^1,$ given $(\phi)\in H_d^{\text{in}}(S_d,C_{\text{pol}})$ (now taken as an isomorphism class of the mere covers, without any Galois-compatibility requirement) there is $\alpha\in\PGL_2(\C)$ such that the critical value where $(\phi)^{\alpha}$ has a $d$-cycle is $\infty$. But a map with a $d$-cycle at $\infty$ has a pole of order $d$ at its only preimage $t_1$, so it is a polynomial in a local coordinate with pole at $t_1$ like $(t-t_1)^{-1}$, and the second part of the claim follows.
\end{proof}
In other words, Lemma \ref{prec} tells us that once we pass to the reduced space we are, up to equivalences, dealing just with polynomials. The choice of not focusing directly on absolute Hurwitz spaces is motivated by our, albeit short, treatment of the issue of lifting rational points, which the literature usually covers in the inner case, in virtue of the connection to the inverse Galois problem.  

We refer to $\beta_d\colon H_d^{\text{red}}(S_d,C_{\text{pol}})\longrightarrow\I_d$ as the Hurwitz space of (Morse) polynomials of degree $d$. Let us compute its degree:
\begin{prop}\label{pol}
   Let $d\ge3$. Then $\beta_d$ has degree $d^{d-3}$, and it is unramified above the $\PGL_2(\C)$-orbits with generic stabilizer.
\end{prop}
\begin{proof}
    By Lemma \ref{prec}, $\deg(\beta_d)$ is equal to the number of isomorphism classes of polynomials in a generic fiber. Since these clearly biject with topological types, there are $d^{d-3}$ of them above each $\{y_1,...,y_d\}\in\U_d$ by Proposition \ref{dd3}. Therefore, it suffices to show that the stabilizer for the $\PGL_2(\C)$-outer action of the class $(f)$ of a polynomial is generically trivial. 
    
    For $d=3$ there is nothing to prove (in particular, our Hurwitz space is a rational point). For $d>4$, the stabilizer $\text{Stab}_{\PGL_2(\C)}(\text{Branch}(f))$ is itself generically trivial, so we are done. For $d=4$, the former stabilizer is generically a Klein four-group, so each nontrivial element $\tau$ is a double transposition; in particular, the critical value $\infty$ is not fixed, so the $4$-cycle in the local monodromy of $\tau(f)$ is not at $\infty$, and therefore $\tau(f)\notin (f).$
\end{proof}

As we just observed, a point $x\in H_d^{\text{red}}(S_d,C_{\text{pol}})(\C)$ always lifts to a complex polynomial $f$: from now on, we will use the expression ``topological type" also for $(f)$, whenever this does not create ambiguity. In order to shed light on the equicriticality problem, we need to be able to relate the minimal field of definition of $f$ to that of $x$; our spaces have the following lifting property:
\begin{teo}Let $K$ be a number field;
\begin{enumerate}[(i)]
\vspace{-2mm}
    \item any $x\in  H_4^{\text{red}}(S_4,C_{\text{pol}})(K)$ lifts to a quartic defined over $K$, so $H_4^{\text{red}}(S_4,C_{\text{pol}})(K)$ is in bijection with the set of $\PGL_2(\C)$-orbits of topological types of quartics over $K$;
    \item if $d>4$, any $x\in  H_d^{\text{red}}(S_d,C_{\text{pol}})(K)$ in the generic fiber of $\beta_d$ lifts to a polynomial defined over a $K$-extension of degree at most $d!$.
\end{enumerate}\label{lifts}
\end{teo}
\begin{proof}
    \begin{enumerate}[(i)]
        \item For the moment, we only prove the claim in the case of nonelliptic critical $j$-invariant, i.e$.$ $\beta_4(x)\notin\{0,1728\}$; the remaining cases will be dealt with in \cref{4}. The datum $(4,S_4,C_{\text{pol}})$ satisfies condition (1) in Lemma \ref{redlift} with $K_0=K$ by the aforementioned Lemma 3.12 in \cite{Ca}, but the stabilizers are Klein groups, not trivial. Let us keep the notation of the proof of Lemma \ref{redlift}, but with $f:=\phi$ taken as a polynomial; then, $\alpha_{\sigma}$ must lie in the stabilizer of $Q$, which is a Klein four-group by the nonellipticity hypothesis. Assume $\alpha_{\sigma}$ is not the identity; as we saw in the proof of Proposition \ref{pol}, its action on the critical values is free, and hence $^{\sigma}f$ would have a $2$-cycle at infinity, which is absurd. Therefore $\alpha_{\sigma}$ is the identity, so for any $\sigma\in\gk$ there exists $\gamma_{\sigma}\in\PGL_2(\overline{\Q})$ such that \begin{align}
            ^{\sigma}f=f\circ\gamma_{\sigma}.\label{fmod}
        \end{align} For mere covers of $\P^1$ without automorphisms, \eqref{fmod} implies that $K$ is a field of definition for $f$ (see the Introduction in \cite{DD}; it is \textit{a priori} only a field of moduli, see 2.2 in the same work for the precise definitions), and our covers indeed have no automorphisms since these biject with the center of the Galois group of the cover (see \cite[\S 1]{ch}), which in our case is $Z(S_d)=\{e\}$. Therefore, there is some $\alpha\in\PGL_2(\overline{\Q})$ such that \begin{align}\label{fdef}
        f\circ\alpha\in K(t).
    \end{align} 
    
    If $\alpha$ is affine we are done, so suppose $\alpha=\begin{pmatrix}a & b \\ c & d\end{pmatrix}, \ c\neq0$. Then $\alpha$ can be written as $\beta\alpha'$ with $\beta$ affine and $\alpha'$ having top-left entry $0$. Therefore, up to replacing $f$ with $f\circ\beta$, $\gamma_{\sigma}$ with $\beta^{-1}\gamma_{\sigma}{^\sigma\beta}$ and $\alpha$ with $\alpha'$, which does not affect \eqref{fmod} and \eqref{fdef}, we can take $a=0$, so say $\alpha=\begin{pmatrix}0 & 1 \\ c & d\end{pmatrix}$. Let us now combine \eqref{fmod} and \eqref{fdef}; for any $\sigma\in G_{K}$, we have \begin{align}\label{sides}
        f=(f\circ\alpha)\circ\alpha^{-1}=\ ^{\sigma}(f\circ\alpha)\circ\alpha^{-1}=f\circ(\gamma_{\sigma}({^{\sigma}\alpha})\alpha^{-1}),
    \end{align}
    which implies that the fractional linear transformation on the right-hand side is the identity, since it induces an automorphism of the cover $f(t)-y$.
    Therefore, since $\gamma_{\sigma}$ is affine by \eqref{fmod}, so must be $^{\sigma}\alpha\alpha^{-1}=\begin{pmatrix}1 & 0 \\ -d\frac{^{\sigma}c}{c}+{^{\sigma}}d & \frac{^{\sigma}c}{c}\end{pmatrix}$. 
    
    Looking at the bottom-left coefficient we get $^{\sigma}(\frac{d}{c})=\frac{d}{c}\ \forall \sigma\in G_{K}\Longrightarrow l=\frac{d}{c}\in K$, and hence $f\left(\frac{c^{-1}}{t+l}\right)\in K(t)$ by \eqref{fdef}.
    But now, as the denominator in the variable is $K$-rational itself, simply taking common denominator $(t+l)^4$ shows that the coefficients $a_i$ of $t^i$ in $f$ satisfy $a_ic^{-i}\in K$, so we finally get: \begin{align*}
        f\bigg(\frac{t}{c}\bigg)\in K[t].
    \end{align*}
        \item Since the generic stabilizer of a point in $\U_d$ is trivial for $d>4$, Lemma \ref{redlift} gives us a polynomial lift with field of moduli $K_0$ equal to the field of definition of any lift of $\beta_d(x)$, and by \cite[Lemma 3.10]{Ca} we can take $[K_0:K]\le d!$. The same argument as above gives a polynomial lift defined over $K_0$.
    \end{enumerate}
\end{proof}
\begin{rem}\label{rtt}
    The paragraph after \eqref{fmod} shows that, in defining $\text{EC}_d(K)$, we can generically just take $K$-topological types; that is, to obtain all rational polynomials equivalent to a given Morse rational polynomial $f$, we can simply take its rational topological type. Indeed, if we have $f\circ \lambda\in K[t]$ for some $\lambda\in\text{Aff}(\overline{K})$, then for any $\sigma\in G_K$ we have that $\lambda^{-1}\circ{^{\sigma}}\lambda$ is an automorphism of the cover $f(t)-y$, and hence trivial. Therefore, $\lambda$ is defined over $K$.
\end{rem}

\section{The case of quartics: modularity}\label{4}

\subsection{The Hurwitz space of quartics as a modular curve}\label{4.1}

From now on, $K$ will denote a number field. We consider the case $d=4$: compactifying our reduced Hurwitz space, we get a rational modular curve\begin{align*}
    \beta_4\colon\H_4\longrightarrow X(1),
\end{align*} to which we refer simply as the Hurwitz space of quartics, since the information about those that are non-Morse is carried by the fiber above the cusp. We adopt the notation $j_{\mathrm{CV}}$ for our coordinate on the level $1$ modular curve, corresponding to the critical $j$-invariant of the (classes of) covers above it.

Let us focus on the geometry of our modular curve $\H_4=:\Gamma\backslash\h$. Recall the formula for the genus of quotients of the upper half plane by modular subgroups \cite[p.66 and 68]{ds}:
\begin{align}
    g(X(\Gamma))=1+\frac{[\PSL_2(\Z):\Gamma]}{12}-\frac{\epsilon_2}{4}-\frac{\epsilon_3}{3}-\frac{\epsilon_{\infty}}{2},
\end{align}
where the $\epsilon_i$'s count points of period $2$ and $3$, and cusps, respectively. Since $[\PSL_2(\Z):\Gamma]=\deg\beta_4=4$, the index term gives a contribution of $\frac{1}{3}$. As the genus is nonnegative, $\epsilon_{\infty}\ge1$, and $3$ is coprime with $4$, surely $\epsilon_3=1$. Then necessarily $g(\H_4)=0$. Notice how \textit{a priori} we have the two possibilities $(2,1)$ and $(0,2)$ for $(\epsilon_2,\epsilon_{\infty})$.

Having degree less than $7$, our $\Gamma$ is a congruence subgroup by a result of Wohlfahrt \cite[Theorem 5]{Wo}, so $\H_4$ is a congruence modular curve. By the list of genus $0$ congruence subgroups in the Cummins-Pauli database \cite{genus}, it is either $X_0(3)$ or $4A^0\backslash\h=X(4A^0)$ as a curve over $\C$ (and hence over $\overline{\Q}$). Observe how these groups realize the two aforementioned possibilities for the number of elliptic points above $j_{\mathrm{CV}}=1728$ and of cusps.

To determine whether $\H_4$ is $X_0(3)$ or $X(4A^0)$ we just need to compute $\epsilon_2$. Recall that quadruples $\widehat y=\{y_1,y_2,y_3,y_4\}\in\U_4$ with elliptic $j$-invariant have stabilizer larger than the usual Klein group (in \cite[\S 4.2.2]{Ca}, Cadoret gives explicit descriptions for the Legendre representatives), in particular $D_8$ for $j_{\mathrm{CV}}=1728$: $\H_4$ has no elliptic point above $j_{\mathrm{CV}}=1728$, i.e. $\epsilon_2=0$, precisely if this stabilizer acts freely on topological types in the fiber. Indeed, the elliptic point above $j_{\mathrm{CV}}=0$ corresponds to (the orbit of) the topological type of $p_0(x)=x^4-x$, which is fixed by multiplication by $\omega$. The following result shows that our reduced Hurwitz space is the modular curve $X_0(3):$

\begin{lem}\label{iell}
    The outer $\PGL_2(\C)$-action on polynomial topological types with critical $j$-invariant $1728$ is free. Therefore, $\H_4\simeq X_0(3)$ over $\overline{\Q}$.
\end{lem}
\begin{proof}
    Let $f$ be a polynomial having critical values as in the hypotheses and let $S_f=\text{Stab}_{\PGL_2(\C)}((f))$. Since $j_{\mathrm{CV}}=1728,$ we can suppose without loss of generality that $C_f=\{0,1,-1\}$ by applying an affine transformation to the polynomial: the new stabilizer will simply be the conjugate under this transformation of the old one; moreover, we can set $f(0)=0$ by a change of variable $x\mapsto x+c$, which does not change the topological type. Any element of $S_f$ must fix $\infty$ as the only critical value of $(f)$ with $4$-cycle local monodromy: from the aforementioned description of the stabilizers for Legendre representatives, we get that the only affine nontrivial element in $S_f$ is $z\mapsto-z$. Therefore, if the statement were false, we would have algebraic numbers $a,b$ such that \begin{align}
        -f(x)=f(ax+b),\label{1728}
    \end{align}which immediately gives $a=\zeta_8$ a primitive $8$-th root of unity. As $f$ and $-f$ have the same roots and $0$ is one of them, iteratively substituting in the right-hand side of \eqref{1728} gives: \begin{align}
        0=f(0)=-f(b)=f(b(\zeta_8+1))=-f(b(\zeta_8^2+\zeta_8+1))=f(b(\zeta_8^3+\zeta_8^2+\zeta_8+1)).\label{1728'}
    \end{align} Since all expressions in $\zeta_8$ appearing in \eqref{1728'} are nonzero (notice how this is false if instead of $\zeta_8$ we put $\omega$, which is the corresponding value in the case $j_{\mathrm{CV}}=0$, where we indeed have an elliptic point---see Corollary \ref{fixp}) and pairwise distinct, and $f$ has at most four roots, we get $b=0$, but then $f(x)=x^4$ which does not have distinct critical values. 
\end{proof}
\begin{reml}
    The modular curve $X_0(3)$ can be realized as a reduced Hurwitz space also for a different datum, namely that of the \textit{Lattès maps} of $3$-isogenies of elliptic curves, see \cite[\S 4.2]{Fr1} and \cite{CIR}.
\end{reml}

Notice how Proposition \ref{pol} together with Lemma \ref{iell} imply the following nice result:
\begin{cor}
    The only topological types of quartic Morse polynomials over $\C$ having nontrivial stabilizer for the outer action of $\text{Aff}(\C)$ are those of the form $[ap_0+b]=[ax^4-ax+b], \ a\in\C^*,b\in\C$.\label{fixp}
\end{cor}

At this point, we have a rational curve $\H_4$ that fits into a diagram:

\begin{center}
        \begin{tikzpicture}[>=stealth,->,node distance=2cm]
  \node (X) at (0,2) {$\H_4$};
  \node (Y) at (2,2) {$X_0(3)$};
  \node (P2) at (0,-0.3) {$X(1)$};

  \draw[->] (X) to node[above] {$\psi_4$} (Y);
  \draw[->] (X) to node[left] {$\beta_4$}(P2);
  \draw[->] (Y) to node[right] {$\pi_3$}(P2);
\end{tikzpicture}
    \end{center} where we take $\pi_3$ (not to be confused with one of the maps in the bottom row of the diagram in \cref{3.1}) as a rational map making $X_0(3)$ into the moduli space of isomorphism classes $[(E,C)]$ of elliptic curves $E$ with a cyclic $3$-torsion subgroup $C$ (so that $j(E)=\pi_3([(E,C)])$). \textit{A priori,} the isomorphism $\psi_4$ is only defined over $\overline{\Q}$: one may ask whether it is forced to be defined over $\Q$. The answer is affirmative:

\begin{lem}
    For any curve $C/\Q$ with a rational map $\beta\colon C\longrightarrow X(1)$ and an isomorphism $\psi\colon C\longrightarrow X_0(3)$ sitting in a diagram as above, $\psi$ is defined over $\Q$. In particular, $\H_4\simeq_{\Q}X_0(3)$.\label{qiso}
\end{lem}
\begin{proof}
    As all objects and morphisms are defined over $\Q$ except at most for $\psi,$ if we act on the diagram by an element $\sigma$ of $\gq$ we get another diagram with $^{\sigma}\psi=\sigma\circ\psi\circ\sigma^{-1}$ in place of $\psi$, from which we extract an automorphism $^{\sigma}\psi\circ\psi^{-1}$ of $X_0(3)$ over $X(1)$. Therefore, it suffices to show that the only such automorphism is the identity. But these automorphisms are classified by Galois theory: their set bijects with the quotient by $H=\Gamma_0(3)/\Gamma(3)\simeq\Z/3$ of its normalizer $N$ in $\Gamma(1)/\Gamma(3)\simeq A_4$, since $X(3)\longrightarrow X(1)$ is the Galois closure of $X_0(3)\longrightarrow X(1)$. As $A_4$ has four $3$-Sylow subgroups, we must have $|N|=\frac{12}{4}=3$ and hence $N=H$, so we are done.
\end{proof}
\begin{reml}
    As the proof suggests, this is generally false if in the statement we replace $(X_0(3),\pi_3)$ with a pair $(C_0,\pi)$ having nontrivial automorphisms. For instance, let $C_0$ be an elliptic curve over $\Q$ with $\pi$ the usual $2$-covering of $\P^1$, and let $C$ be a nontrivial quadratic twist; the isomorphism $\alpha$ is then only defined over a quadratic extension of $\Q$, the generator of whose Galois group sends $\alpha$ to $[-1]_{C_0}\circ\alpha$, with $[-1]_{C_0}$ the multiplication by $-1$ on $C_0$.
\end{reml}

\begin{cor}
    We have $X_0(3)\simeq_{\Q}X_1(3)$ and hence, for any elliptic curve $E/\Q$ with a cyclic subgroup of order $3$ defined over $\Q$, there exist an elliptic curve $E'/\Q$ with a rational point of (exact) order $3$ and such that $j(E)=j(E')$.\label{x0x1}
\end{cor}
\begin{proof}
    As $\Gamma_0(3)=\Gamma_1(3)$, we have an isomorphism $\psi:X_1(3)\longrightarrow X_0(3)$ over $\C$, and hence over $\overline{\Q}$, that commutes with the projections to $X(1)$. The claim now follows from Lemma \ref{qiso}.
\end{proof}

\subsection{A direct method}\label{4.2}

Our Hurwitz space $\H_4$ has genus $0$ and a certain algebraic structure over $\Q$ coming from Hurwitz theory. Since a point on $\H_4$ represents a $\PGL_2(\C)$-orbit of topological types, it is natural to ask if we can describe its coordinate explicitly as a function of the orbit. Recall that the morphism $\theta_d$ of Section 2 maps critical points to critical values: one may then guess that $\beta_4$ maps the $j$-invariant of the critical points of the orbit to that of the critical values. Indeed, this map is well-defined: the critical points of a topological type form a $\PGL_2(\C)$-orbit, while $Z_{\alpha(f)}=Z_f$ for $\alpha\in\PGL_2(\C)$ (basically, their behavior under the two $\PGL_2(\C)$-actions we introduced is the mirror image of that of critical values). Moreover, to any $j\in\C$ we can associate the $\PGL_2(\C)$-orbit of the (unique, by the above discussion) topological type of $$\int_0^x(u-u_1)(u-u_2)(u-u_3) du,$$ with $\{u_1,u_2,u_3,\infty\}$ any quadruple with $j$-invariant $j$. In summary, by sending an orbit to the $j$-invariant of its critical points, we get a bijection:
\begin{align}
\left\lbrace
\begin{aligned}
\displaystyle
\text{$\PGL_2(\C)$-orbits of topological types of complex Morse quartics} \
\end{aligned}
\right\rbrace
\overset{1:1}{\longleftrightarrow} \ 
\A^1(\C).   \label{cp}
\end{align}

Identify $H_4^{\text{red}}(S_4,C_{\text{pol}})$ with $\A^1$ via the $j$-invariant $j$ of the critical points: by compactifying, the above correspondence yields a map $\widetilde{\beta_4}\colon\H_4\longrightarrow X(1)$, \ $j\mapsto j_{\mathrm{CV}}$. Note that even if we know that such a map exists, it need not be $\beta_4$ \textit{a priori:} nonetheless, we will now prove that $\widetilde{\beta_4}$ is indeed an algebraic morphism defined over $\Q$ by computing it explicitly, so necessarily $\widetilde{\beta_4}$ and $\beta_4$ are the same up to a change of variable in $\PGL_2(\Q)$ by the uniqueness property for algebraic models of Hurwitz spaces (see the paragraph before \eqref{galc}). In the following, the variable $j$ will represent a point in $\H_4$ with the above identification.

Before working out an explicit model for $\beta_4$, we remark that the above discussion shows how \textit{any} point $j\in H_4^{\text{red}}(S_4,C_{\text{pol}})(K)$ lifts to a quartic defined over $K$: we can just integrate the Weierstrass polynomial of any $K$-rational elliptic curve with $j$-invariant $j$ (the existence of such a curve is well known, see \cite[Proposition 1.4]{silv}). This completes the proof of Theorem \ref{lifts}.

Fix an isomorphism $u\colon\P^1\longrightarrow X_0(3)$ and let us take for $\pi_3$ the model given by $u\mapsto \frac{(u+3)^3(u+27)}{u}$ (see \cite[p.23]{su}). For a complex cubic curve given by a Weierstrass equation $E: y^2=x^3+Ax+B$, let\footnote{our choice of linear normalization for the integral turns out to be convenient in computing the transformation in \eqref{2map}, but any other choice would work.}: \begin{align}
    f_{E}(x)=12\int_0^x(s^3+As+B) \ ds - A^2.\label{normp}
\end{align} Moreover, let $\Delta_E=-16(4A^3+27B^2)$ be the discriminant of $E$ (note that we are allowing $\Delta_E$ to vanish), let $j=-1728\frac{(4A)^3}{\Delta_E}\in\P^1(\C)$ be its $j$-invariant and $j_{\mathrm{CV}}\in\P^1(\C)$ be that of the cubic $E^{\mathrm{CV}}:z^2=p_E(x)$, where $p_E(x)$ is the monic polynomial vanishing at the critical values of $f_{E}$. Since $j$ and $j_{\mathrm{CV}}$ represent respectively the $j$-invariant of the critical points and of the critical values of a variable quartic, there is no ambiguity with our coordinates for $\H_4$ and $X(1)$, as the latter stand for the same things in our moduli problem. 
\begin{lem}\label{a4}
    We have \begin{align*}
    j_{\mathrm{CV}}=\pi_3\left(\frac{j}{64}-27\right).
\end{align*}
\end{lem}
\begin{proof}
The roots of $p_E$ are $f_E(x_i)$, for $x_i$ the three roots of the Weierstrass polynomial of $E$. Looking at $f_E$ modulo the latter, we find that the roots of $p_E$ are $r_i=3Ax_i^2+9Bx_i-A^2$. 
Since
\begin{align*}
x_1+x_2+x_3=0,\\
x_1x_2+x_1x_3+x_2x_3=A,\\
x_1x_2x_3=-B,
\end{align*}
we get
\begin{align*}
\sum r_i=-9A^2,\quad
\sum_{i<j} r_ir_j=-\frac{3}{8}A\Delta_E,\quad
\prod r_i=\frac{\Delta_E^2}{256};
\end{align*}
therefore,
\begin{align}
E^{\mathrm{CV}}:y^2=x^3+\left(3Ax-\frac{\Delta_E}{16}\right)^2.\label{eq: ecv}
\end{align}
\noindent
Applying the $j$-invariant formula from \cite[III.1]{silv} gives:
\begin{align*}j_{\mathrm{CV}}=-\frac{256A^3(A^3-54B^2)^3}{B^2(4A^3+27B^2)^3}.
\end{align*}
Letting
\begin{align*}
\mu:=\frac{27B^2}{4A^3}
=\frac{1728-j}{j},
\end{align*}
we get:
\begin{align}
j_{\mathrm{CV}}=-\frac{27(1-8\mu)^3}{\mu(1+\mu)^3}=-\frac{27\left(9\frac{j-1536}{j}\right)^3}
{\left(\frac{1728-j}{j}\right)\left(\frac{1728}{j}\right)^3}=\frac{j(j-1536)^3}{2^{18}(j-1728)}=\pi_3\left(\frac{j}{64}-27\right).\label{b4model}
\end{align}
\end{proof}

\noindent
Observe how the second-to-last expression in \eqref{b4model} is our explicit model for $\beta_4$, so we find $\psi_4(j)=\frac{j}{64}-27$ in the diagram in \cref{4.1}. Moreover, Lemma \ref{a4} implies that the integrals of the Weierstrass polynomials of $\C-$isomorphic elliptic curves have the same critical $j$-invariant, which is not a surprise since it is equivalent to the right direction of \eqref{cp}. Finally, the lemma allows us to characterize the critical $j$-invariants of quartics over $K$: let \begin{align*}
        CJ(K)=\{v\in K\cup\{\infty\}:\exists f\in K[x]\text{ s.t. $\deg f=4$ and }j(C_f)=v\}.
    \end{align*} 
\begin{prop}
    We have:
    \begin{align*}
       CJ(K)=\biggl\{\frac{(u+3)^3(u+27)}{u}, \ u\in K\cup\{\infty\}\biggl\}.
    \end{align*} 
  \label{eqj}
\end{prop}
\vspace{-6mm}
\begin{proof}
    Theorem \ref{lifts}(i) and Lemma \ref{qiso} imply that \begin{align}
        \text{CJ}(K)\cup\{\beta_4(\infty)\}=\{\infty\}\cup\biggl\{\frac{(u+3)^3(u+27)}{u}, \ u\in K\cup\{\infty\}\biggl\}=\biggl\{\frac{(u+3)^3(u+27)}{u}, \ u\in K\cup\{\infty\}\biggl\}
    \end{align}(the singleton in the intermediate expression is added \textit{a priori} to make sure we are accounting for non-Morse quartics), while Lemma \ref{a4} tells us that $\beta_4(\infty)=\infty=\beta_4(1728)\in\text{CJ}(K)$, so we are done.\end{proof}

Let us examine the fiber of $\beta_4$ above $j_{\mathrm{CV}}=0$: since $\pi_3$ vanishes only for $u=-3$ and $u=-27$ (with order respectively $3$ and $1$, in accordance with the ramification structure we previously found), Lemma \ref{a4} tells us that the values of $j$ corresponding to these two preimages are $0$ and $1536$: in other words, the elliptic curves with $j$-invariant $0$ and $1536$ are precisely those with Weierstrass integrals having critical $j$-invariant $0$. What about elliptic curves with $j=1728$? Since $\beta_4(1728)=\infty$, they have non-Morse Weierstrass integrals: this is indeed immediate to check with $E:y^2=x^3-x$. 

We are now almost ready to prove Theorem \ref{th2}: we only need the following strengthening of Theorem \ref{lifts}\\r(i):
\begin{lem}\label{polrep}
    Let $j_{\mathrm{CV}}\in K\setminus\{0,1728\}$. For any $j\in\H_4(K)$ such that $\beta_4(j)=j_{\mathrm{CV}}$ and any quadruple $\widehat y=\{y_1,y_2,y_3,\infty\}\in\U_4(K)$ with $j(\widehat y)=j_{\mathrm{CV}}$, $j$ lifts to a quartic $f\in K[x]$ with $C_f\cup\{\infty\}=\widehat y$.\end{lem}

\begin{proof}
Let $f_0$ be a lift of $j$ as in the statement of Theorem \ref{lifts}, and set $\widehat w=C_{f_0}\cup\{\infty\}$. As $j(\widehat w)=j_{\mathrm{CV}}$ by construction, there is $\lambda\in\PGL_2(\overline{\Q})$ such that $\lambda(\widehat w)=\widehat y$ (as sets); up to composing with the double transposition in $\text{Stab}_{\PGL_2(\overline{\Q})}(\widehat y)$ that swaps $\infty$ and $\lambda(\infty)$, we can assume $\lambda$ to be affine. Therefore, $\lambda(f_0)$ is a quartic with critical values $\widehat y$, but it may not be defined over $K$. Still, it has nonelliptic critical $j$-invariant and it satisfies equation \eqref{fmod} for any $\sigma\in\gk$, so by the argument after that equation its topological type contains a rational polynomial.
\end{proof}

\begin{proof}[Proof of Theorem \ref{th2}]
    Let us adopt the notation from the statement of the theorem. For $j(E)\notin\{0,1728\}$ the claim follows immediately from Proposition \ref{eqj} and Lemma \ref{polrep}. For $E/K$ with $j(E)=0,$ we show that all such triples $\{y_1,y_2,y_3\}$ can be realized as the finite critical values of a quartic $f\in K[x]$ with $j(Z_f)=0$; given such $f$, from \eqref{eq: ecv} and $A=0$ one gets the model $y^2=x^3+729B^4$ for $E^{\mathrm{CV}}$. Multiplying $f$ by $d\in K^*$ simply twists $E^{\mathrm{CV}}$ by $d$, so we can realize all curves \begin{align}y^2=x^3+729B^4d^3\label{eq:j=0}\end{align} as $E^{\mathrm{CV}}$, for $B,d\in K^*$. Recall that $K$-isomorphism classes in the family of $E/K$ with $j(E)=0$ are classified by $K^*/(K^*)^6$, as the class of the constant term of a Weierstrass polynomial. Since $729=3^6$ and the map $\Z/6\times\Z/6\to\Z/6, \ (m,n)\mapsto4m+3n$ is surjective, we are done in virtue of \eqref{eq:j=0}.

    For $j_{\mathrm{CV}}=1728,$ one finds from \eqref{b4model} that the two preimages in $\H_4$ are $j=1152\pm384\sqrt3$ (see Remark \ref{nongen} for more detail). Take $f$ with $j(Z_f)=j$: this is equivalent to $27B^2=(5\mp3\sqrt3)A^3$; substituting this in \eqref{eq: ecv} after shifting $x=X-3A^2$ gives the model \begin{align}
        y^2=x^3+9(1\pm2\sqrt3)A^4x
    \end{align} for $E^{\mathrm{CV}}$. Since $K$-isomorphism classes of the family of $E/K$ with $j(E)=1728$ are classified by $K^*/(K^*)^4$, as the class of the coefficient of $x$ in a Weierstrass polynomial, multiplying $f$ by a constant we obtain exactly the two quadratic twist subfamilies of the statement (note that shifting $f$ by an additive constant does not change the $K$-isomorphism class of $E^{\mathrm{CV}}$).
\end{proof}

\section{Equicritical quartics and elliptic curves}\label{5}

We have come to the question of inequivalent equicritical quartics: the goal of this section is to prove Theorem \ref{mainthm}. Let $\mathcal{C}_4$ be the non-diagonal irreducible component of $\H_4\underset{X(1)}{\times}\H_4$ (there may be \textit{a priori} up to three such components, but the results of \cref{51} will show that there is only one) and let $\nu:C_4\longrightarrow\mathcal{C}_4$ be its normalization; we refer to a point of $P\in C_4(\overline{\Q})$ such that $\nu(P)$ is singular (or to a pair of quartics to which $P$ lifts) as ``exceptional". A pair of inequivalent equicritical quartics $(f,g)$ over a number field $K$ clearly defines a $K$-point on $C_4$. Moreover, Lemma \ref{polrep} implies that, generically, such a point lifts to a pair of $K$-rational equicritical quartics. Therefore, our task is reduced, at least generically, to that of finding rational points on $C_4$.

\subsection{Elliptic curves with constant mod $3$ representation}\label{51}

As $H_4^{\text{red}}(S_4,C_{\text{pol}})\simeq_{\Q}Y_0(3)$ as covers of $Y(1)$ by Lemma \ref{qiso}, $C_4$ turns out to be $\Q$-isomorphic to the curve $X_3$ studied by Rubin and Silverberg \cite{RS} in the context of parametrizing elliptic families with given mod $p$ Galois representation. Indeed, in virtue of Corollary \ref{x0x1} we have \begin{align*}
    H_4^{\text{red}}(S_4,C_{\text{pol}})\underset{Y(1)}{\times}H_4^{\text{red}}(S_4,C_{\text{pol}})\simeq Y_0(3)\underset{Y(1)}{\times}Y_1(3),
\end{align*}so removing the diagonal we precisely obtain the second equivalent definition for $Y_3$ given in \cite[\S 1]{RS}. This is enough for us to give a proof of a nonconstructive version of Theorem \ref{mainthm}:
\begin{lem}
    There exists a quasi-basis for $\text{EC}_4(K)$ generically parametrized by $\P^1(K)$.\label{noncp}
\end{lem}
\begin{proof}
    Observe that any collection of pairs of $K$-rational equicritical quartics $\{(f_P,g_P), \ P\in C_4(K)\}$ parametrized by $C_4(K)$ in the above sense (where we take such a pair for points above $j_{\mathrm{CV}}\in\{0,1728,\infty\}$ only if it exists) makes up a quasi-basis for $EC_4(K)$. Since the genus of $X_3$ is $0$ and it has rational points (see \cite[\S 1.1]{RS}), the claim follows. Alternatively, if one did not know about the modular characterization of the fiber product, the genus can be computed via Riemann-Hurwitz for the desingularization, and the rational point $(0,1536)$ can be exhibited. Even without knowing the explicit form for $\beta_4$ we worked out in \cref{4}, but just the \textit{geometry} of the rational map $\beta_4\colon\H_4\longrightarrow X(1)$, one can exhibit the same rational point by observing that, since the fiber $\beta_4^{-1}(0)$ has two elements and one of them is rational (it corresponds to the orbit of $(p_0)$), the other one must be too.
\end{proof}

The curve $X_3$ is itself a moduli space of elliptic curves, with the additional structure induced by the self-fiber product of $X_0(3)$. In the aforementioned work of Rubin and Silverberg, the authors give a certain rational model $t\colon\P^1\longrightarrow X_3$ of $X_3$, and write down the respective affine equation for the universal elliptic curve over $X_3$ as 
\begin{align*}
    E_t\colon Y^2=X^3-27t(t^3+8)X+54(t^6-20t^3-8),
\end{align*}having $j$-invariant
\begin{align}
    j_t=j(E_t)=27\left(\frac{t(t^3+8)}{t^3-1}\right)^3.\label{juni}
\end{align}

We take as our model for $X_3$ the one obtained by compactifying the self-fiber product of $(X_0(3),\pi_3)$. Thanks to Lemma \ref{a4} and \eqref{juni}, we can explicitly parametrize $C_4(K)$: this amounts to finding two rational functions $x_1(t), x_2(t)$ such that the pairs $(x_1(t),x_2(t))$ as $t$ ranges in $K\cup\{\infty\}$ parametrize $X_3(K)$, and then pulling them back via $\psi_4$.
Searching for rational functions such that \begin{align}
    \pi_3(x(t))=j_t\label{sect}
\end{align}by comparing the right-hand side of \eqref{juni} with $\pi_3$, it is not particularly hard to figure out that \begin{align}
    x_1(t)=\frac{27}{(t^3-1)}\text{ and }x_2(t)=\frac{3(t-1)^3}{t^2+t+1}=x_1\left(\frac{t+2}{t-1}\right)\label{xis}
\end{align}
work. We then have:
\begin{prop}
     The multiset $\{(\psi_4^{-1}(x_1(t)),\psi_4^{-1}(x_2(t))), \ t\in K\cup\{\infty\}\}$ parametrizes $C_4(K)$.\label{magic}
\end{prop}

\begin{proof}    
    It suffices to show that $X_3(K)$ is parametrized by $\{(x_1(t),x_2(t)), \ t\in K\cup\{\infty\}\}$. One easily verifies that \eqref{sect} holds for $x(t)=x_i(t), \ i=1,2$. Now, since $x_1(s)=x_1(t)\iff t\in{\langle\omega\rangle}s$, \eqref{xis} implies that $x_1(t)\neq x_2(t)$ generically, so $t\mapsto(x_1(t),x_2(t))$ defines a morphism $\chi\colon\P^1\longrightarrow X_3$. We just need to show that its degree is $1$. 
    
    Again by \eqref{xis}, this amounts to verifying that $x_2(\omega^k t)= x_2(t)$ does not hold identically for $k=1,2$. Rewriting $x_2(t)$ as $3\frac{(t-1)^4}{t^3-1}$, we see that equality holds when either $t\in\{\omega,\omega^2\}$ or $(w^k t-1)^4=(t-1)^4$. Since the polynomials appearing at the left and right-hand sides are distinct, the claim follows. Let us remark that, since $\gamma(t)=\frac{t+2}{t-1}$ is an involution, \eqref{xis} implies that $\chi(t)$ and $\chi(\gamma(t))$ give the same unordered pair.
\end{proof}
    
\begin{reml}
    By solving $x_1(t)=x_2(t)$, we find that for $t\in\{\omega,\omega^2,-2\omega,-2\omega^2,1\pm\sqrt3\}$ we obtain ``diagonal" points on $C_4$, specifically $(\infty,\infty)$ for the first two values, $(1536, 1536)$ for the second two and $(1152\pm384\sqrt3,1152\pm384\sqrt3)$ for the last two (as expected, their coordinates are the four ramification points of $\beta_4$). It is important to keep in mind that we are parametrizing $C_4$ with a multiset, so in the first two cases we are dealing with two distinct points, even if they have the same coordinates. In particular, $C_4$ has two exceptional points above each of $j=\infty$ and $j=1536$. But there are more: solving $\chi(s)=\chi(t)$, we find that the values $t=\omega(1\pm\sqrt3)$ yield the remaining ones, satisfying respectively $\chi(t)=\chi(\omega t)=(1152\pm384\sqrt3,1152\mp384\sqrt3)$.\label{nongen} 
\end{reml}
\begin{reml}
    The map $X_3\longrightarrow X(1), \ (x_1(t),x_2(t))\mapsto j_t$ has degree 12. One easily verifies that the preimages of $j_t$ are given by \begin{align*}
        \left(x_1(\gamma^i(w^k\gamma^j(t))), \ x_2(\gamma^i(w^k\gamma^j(t)))\right), \ 0\le i,j\le1, \ 0\le k\le2,
    \end{align*}so, given $t\in K$, there are four points defined over $K$ that map to $j_t$ if and only if $\omega\in K$.\label{quadr}
\end{reml}
Proposition \ref{magic} and \eqref{xis} give: \begin{align}
    C_4(K)=\biggl\{\left(1728\frac{t^3}{t^3-1},1728+192\frac{(t-1)^3}{t^2+t+1}\right), \ t\in K\cup\{\infty\}\biggl\},\label{jequicq}
\end{align}
where we treat pairs appearing with multiplicity as distinct points. Recall the notation $\rho=1+\sqrt3, \ \overline{\rho}=1-\sqrt3$, let $\widetilde K=K\setminus\{\omega,\omega^2,-2\omega,-2\omega^2,\rho,\overline{\rho}\}$ and set $j_1(t)=1728\frac{t^3}{t^3-1}$, $j_2(t)=1728+192\frac{(t-1)^3}{t^2+t+1}$. Then \eqref{jequicq} means precisely that, for $t\in K\cup\{\infty\}$, the integrals of the Weierstrass polynomials of any two cubic curves $E_t, F_t$ with $j(E_t)=j_1(t)$ and $j(F_t)=j_2(t)$ are two quartics with critical $j$-invariant $j_t$, and they are inequivalent if $t\in\widetilde K\cup\{\infty\}$. This, along with Lemma \ref{polrep}, gives a nice criterion for a (generic) quartic over $K$ to have an inequivalent equicritical ``twin":
\begin{cor}
    Let $f\in K[x]$ be a Morse quartic with critical $j$-invariant $j_{\mathrm{CV}}\neq0,1728$. Then, there exists a quartic $g\in K[x]$ with $C_f=C_g$ and $[f]\neq[g]$ if and only if the $j$-invariant of the elliptic curve $y^2=f'(x)$ is of the form $1728\frac{t^3}{t^3-1}$ for some $t\in K$.
\end{cor}
\begin{reml}
    The hypotheses that $f$ is Morse and $j_{\mathrm{CV}}$ is nonelliptic imply, respectively, that $y^2=f'(x)$ is an elliptic curve and that Lemma \ref{polrep} applies.
\end{reml}

In order to parametrize a basis for $\text{EC}_4(K)$, we need to find an equicritical pair over $K$ for each $t$ (if there is one), that is, we need to explicitly go through the procedure described in the proof of Lemma \ref{polrep} for nonelliptic $j_{\mathrm{CV}}=j_t$, while taking care of the excluded values separately. Specifically, given $t\in \widetilde K\setminus\{-2,0,1\}$, we:
\begin{enumerate}
    \item pick elliptic curves $E_t,F_t$ over $K$ with $j(E_t)=j_1(t), \ j(F_t)=j_2(t)$, yielding us quartic lifts $f^t=f_{E_t}, 
\ g^t=f_{F_t}\in K[x]$ of $j_1(t)$ and $j_2(t)$, with the notation of \eqref{normp};
    \item find the lift $\widetilde{g^t}$ of $j_2(t)$ such that $C_{\widetilde{g^t}}=C_{f^t}$, by computing $\lambda_t(g^t)$ with $\lambda_t\in\text{Aff}(K)$ the transformation mapping $F^{\mathrm{CV}}_t[2]$ to $E^{\mathrm{CV}}_t[2]$ (we will not need to pick another representative of the topological type since, as we will see, the family $\lambda_t$ is defined over $\Q(t)$).
\end{enumerate} 

The first step is easy: it is well known that the elliptic curve $E_j:y^2=x^3+\frac{3j}{1728-j}x+\frac{2j}{1728-j}$ has $j$-invariant $j$ for all $j\neq0,1728$, so we can just take
\begin{align*}
    E_t: y^2=x^3-3t^3x-2t^3\text{ and } F_t: y^2=x^3-3{\left(\frac{t+2}{t-1}\right)^3}x-2\left(\frac{t+2}{t-1}\right)^3, \ t\in\widetilde K\setminus\{-2,0,1\},
\end{align*}
corresponding to \begin{align}
    f^t(x)=3x^4-18t^3x^2-24t^3x-9t^6\text{ and }g^t(x)=3x^4-18{\left(\frac{t+2}{t-1}\right)^3}x^2-24{\left(\frac{t+2}{t-1}\right)^3}x-9\left(\frac{t+2}{t-1}\right)^6.\label{pols}
\end{align}The second step is in principle more complex, but a careful examination of the critical values for the first few negative integer values for $t$ suggests for $\lambda_t$ the form \begin{align}
    z\mapsto-\frac{t^4(t-1)^6}{3(t+2)^4}z-36t^4(t^2+t+1),\label{2map}
\end{align} which we easily verify with \textbf{SageMath} \cite{sagemath}. As anticipated, the $\lambda_t$'s turn out to be defined over $K$, so there is no need to look for a rational representative in $[g^t]$. Let us remark that it is always the case that, for two elliptic curves $E,F$ defined over $K$, isomorphic over $\overline\Q$ and with nonelliptic $j$-invariant, their $2$-torsions are related by an affine transformation defined over $K$. Indeed, any $\overline{\Q}$-isomorphism $\phi:E\longrightarrow F$ descends to an isomorphism $\lambda:E[2]\longrightarrow F[2]$ of the $2$-torsions as Galois modules, so in particular it is affine. If $\lambda$ is not defined over $K$, there is $\sigma\in G_K$ such that $\lambda^{-1}\circ{^{\sigma}}\lambda\neq I$, so in particular $\phi^{-1}\circ{^{\sigma}}\phi\in\Aut_{\overline{\Q}}(E)\setminus\langle\iota\rangle$, with $\iota$ the elliptic involution; in other words, $E$ has nontrivial automorphisms, and hence elliptic $j$-invariant.

Putting together Proposition \ref{magic}, \eqref{pols} and \eqref{2map}, we have proved:
\begin{lem}
    The collection of pairs of Morse polynomials $\e'_K=(\c_t=(f_t,g_t), t\in \widetilde K\setminus\{-2,0,1\})$ with \begin{align}\begin{split}
    &f_t=\frac{1}{3t^4}f^t(tx)+3t^2=x^4-6tx^2-8x,\\
    &g_t=\frac{1}{3t^4}\lambda_t\left(g^t\left(\frac{t+2}{t-1}x\right)\right)+3t^2=-\frac{(t-1)^2}{3}\left(x^4-6\frac{t+2}{t-1}x^2-8x\right)-8(t^2+t+1)\end{split}\label{pairs}
\end{align}
is a quasi-basis for $EC_4(K)$.\label{genericf}
\end{lem}
More specifically, substituting in \eqref{jequicq} we see that the values $t=-2,0,1,\infty$ we left out correspond to the pairs $(1536,0),(0,1536),(\infty,1728),(1728,\infty)\in C_4(K)$, while Remark \ref{nongen} tells us that the other points in $C_4(K)$ which we might be missing representatives for are the eight exceptional ones, along with the two diagonal points $(1152\pm384\sqrt3,1152\pm384\sqrt3)$.

\subsection{Elliptic fibers and cusps}
\label{52}
Let us now deal with equicritical pairs lifting the remaining eight points in $C_4(K)$. We start by examining the situation above $j_{\mathrm{CV}}\in\{0,1728\}$.\begin{enumerate}
    \item $j_{\mathrm{CV}}=0$: let us deal with the points $(0,1536)$ and $(1536,0)$ first; 
    we already found a representative $(p_0)$ for the $\PGL_2(\C)$-orbit of topological types corresponding to $j=0$, and we can do the same for $j=1536$ simply by integrating the Weierstrass polynomial of any elliptic curve with $j$-invariant $1536$. Still, the affine transformation relating the critical values of the two representatives is not guaranteed to be rational (recall that this is why Lemma \ref{polrep} fails for elliptic $j_{\mathrm{CV}}$), but it happens to be so in some cases: choosing the right elliptic curve with the help of the L-functions
    and modular forms database (LMFDB) \cite{lmfdb}, we obtain the two additional equicritical pairs $\c_0=(f_0,g_0)$ and $\c_{-2}=(g_0,f_0)$ with \begin{align}
        f_0=p_0=x^4-x, \ g_0=-\frac{1}{48}x^4-\frac{1}{4}x^2+\frac{1}{6}x-\frac{1}{2}.\label{j=0}
    \end{align}
    
    Let us now look at the two exceptional points in this fiber, both represented by the coordinates $(1536,1536)$ in our parametrization: we know that, for any $\widehat y\in\U_4(K)$ such that $j(\widehat y)=j_{\mathrm{CV}}$, there are three topological types with critical values $\widehat y$ that lift $j=1536$. We have just constructed a rational one $[g_0]$; the other two will be given by $[\lambda(g_0)]$ and $[\lambda^{(2)}(g_0)]$ with $\lambda$ the nontrivial affine transformation preserving $C_{g_0}$ (we know by Corollary \ref{fixp} that these are distinct topological types), which, since $C_{g_0}=C_{p_0}$, is simply the multiplication by $\omega$. As $\Q(\omega)$ is clearly the minimal field of definition of any conjugate of $\lambda$, we get two additional pairs \begin{align}
        \c_{-2\omega}=(g_0,\omega g_0)=\overline{\c_{-2\omega^2}}\label{j0double}
    \end{align}if and only if $\omega\in K$ (and no other, since the pairs $(g_0,\omega^2 g_0)$ and $(\omega g_0,\omega^2 g_0)$ belong to the orbit of one of those appearing in \eqref{j0double}).
    \vspace{1mm}
    
    \item $j_{\mathrm{CV}}=1728$: we consider the diagonal points $(1152\pm384\sqrt3,1152\pm384\sqrt3)$ first; they correspond to the values $t=\rho,\overline{\rho}$ in our parametrization, so the quartics $f_{\rho},f_{\overline{\rho}}$ (following the notation of \cref{51}) are lifts for $j=1152+384\sqrt3$ and $j=1152-384\sqrt3$ respectively. As we just saw in the case $j_{\mathrm{CV}}=0$, to find the other lifts making up the pairs we need to apply to $f_{\rho}$ and $f_{\overline{\rho}}$ the affine transformations preserving their respective critical value sets. \textbf{SageMath} \cite{sagemath} shows that $f_{\rho}=x^4-6\rho^3x^2-8\rho^3x$ satisfies $C_{f_{\rho}}=(-\upsilon+C,C,\upsilon+C)$ for $\upsilon=72\sqrt{362\sqrt3 + 627}$ and $C=-720\sqrt3 - 1248$. As $f_{\rho}$ and $f_{\overline{\rho}}$ are Galois conjugates, so are $C_{f_{\rho}}$ and $C_{f_{\overline{\rho}}}$, and hence we get the additional pairs
    \begin{align} 
        \c_{\rho}=(f_{\rho}, \ -f_{\rho}+2C), \ \c_{\overline{\rho}}=(f_{\overline{\rho}}, -f_{\overline{\rho}}+2\overline{C})\label{j=1728}
    \end{align}if and only if $\sqrt3\in K$. 
    
    The previous paragraph enables us to also find pairs lifting the four exceptional points in this fiber, corresponding to $t=\omega\rho,\omega^2\rho,\omega\overline{\rho},\omega^2\overline{\rho}$, as their orbits in $\text{EC}_4(K)$ are necessarily obtained by taking one orbit from each of the pairs in \eqref{j=1728}. Again with the help of \textbf{SageMath} \cite{sagemath}, we find that the linear map sending $C_{f_{\overline{\rho}}}$ to $C_{f_{\rho}}$ is $z\mapsto iRz-iR\overline{C}$, with $R=362+209\sqrt3$, yielding the four additional pairs\begin{align}
        \c_{\omega\rho}=(f_{\rho}, \ iRf_{\overline{\rho}}-iR\overline{C})=\overline{\c_{\omega\overline{\rho}}}, \ \c_{\omega^2\rho}=(f_{\rho}, \ iR(-f_{\overline{\rho}}+2\overline{C})-iR\overline{C})=\overline{\c_{\omega^2\overline{\rho}}}\label{exc1728}
    \end{align} if and only if $\omega,\sqrt3\in K$.
    Note that the other four possible combinations do not appear in the basis, as they all belong to the $\PGL_2(\C)$-orbit of one of those in \eqref{exc1728}.
\end{enumerate}

Let us now look above $j_{\mathrm{CV}}=\infty$. As anticipated, the two cusps $j=\infty,1728$ of $\H_4$ correspond to the $\PGL_2(\C)$-orbits of topological types $(f)$ of non-Morse quartics; indeed, there are two of these: if $f$ branches over less than three points, then either $f'$ has a double root---in which case $f$ is of the form $A(x-a)^3(x-b)+B, \ (A,B)\in K^*\times K, \ a,b\in\overline{\Q}$ (\textit{type 1})---or $f$ takes the same value at two of the distinct roots of $f'$, i.e. $f=A(x-a)^2(x-b)^2+B$ for $A,B,a,b$ as above (\textit{type 2}) (we can assume $a\neq b$, as we otherwise get the $\text{Aff}(K)$-orbit of a fourth power, which is, above each choice for the critical value $B$, the class of a unique topological type).
When we mod out by the outer action of $\text{Aff}(\overline{\Q})$, \textit{type 1} and \textit{type 2} are clearly distinct topological types. Observe that if a multiset $\{w,w,x\}$ is fixed by an affine transformation, this must necessarily fix $w$ as the element of multiplicity $2$, and hence also $x$, so it must be the identity: therefore, there are no equicritical non-Morse quartics arising from the same outer $\text{Aff}(\overline{\Q})$-orbit. In other words, the two exceptional points with coordinates $(\infty,\infty)$ in our parametrization, which correspond to $t=\omega,\omega^2$, do not lift to an inequivalent equicritical pair. 

All we have to do now is find, if it exists, an equicritical pair $(f_1,g_1)$ with $f$ of \textit{type 1} and $g$ of \textit{type 2}. It is not hard to see that \begin{align}
    (f_1,g_1)=(-3x^3(x+2), \ x^2(x+3)^2),\label{j=inf}
\end{align}with critical values $\{0,0,\frac{81}{16}\}$, works, and this gives the last two ordered pairs $\c_1=(f_1,g_1)$ and $\c_{\infty}=(g_1,f_1)$. Finally, the parametrization part of Theorem \ref{mainthm} follows from Lemma \ref{genericf} together with \eqref{j=0}, \eqref{j0double}, \eqref{j=1728}, \eqref{exc1728} and \eqref{j=inf}, while the last part follows from Remark \ref{quadr}.

\subsection{Application to Weyl sums mod $p^2$, and one example}\label{5.3}

As we anticipated in \cref{1}, our main result has an application in the theory of exponential sums attached to polynomials. Specifically, given $a\in\Z$ and a prime power $q$, in \cite{ks} Kowalski and Soundararajan consider the mod $q$ Weyl sum\begin{align*}
    W_{f}(a,q)=\frac{1}{\sqrt q}\sum_{x\in\Z/q\Z}e\left(\frac{af(x)}{q}\right)
\end{align*}attached to $f\in\Z[x]$, investigating pairs $(f,g)$ satisfying $W_f(a,q)=W_g(a,q)$ (or other weaker variants). For $q=p^2$ they observe \cite[Remark 1.10]{ks} that, when $(a,p)=1$, one has\begin{align*}
    W_{f}(a,q)=\sum_{\substack{f'(v)=0\\v\in\F_p}}e\left(\frac{af(v)}{p}\right),
\end{align*} so equicritical polynomials $f,g\in\Z[x]$ whose leading coefficients are invertible mod $p$ satisfy $W_f(a,p^2)=W_g(a,p^2)$ for $(a,p)=1$. Therefore, following its notation, Theorem \ref{mainthm} implies:
\begin{cor}
    Let $p>3$ be a prime. For any $t\in\Z$ such that $p\nmid(t-1)$, the integral quartic polynomials $f(x)=3f_t(x)$ and $g(x)=3g_t(x)$ are linearly inequivalent and satisfy $W_f(a,p^2)=W_g(a,p^2)$ whenever $a$ is coprime to $p$.
\end{cor}

Finally, let us test Theorem \ref{mainthm} with an example: take, for instance, $t=42$. Then, the simple implementation 
\begin{verbatim}
from sage.all import QQbar, polygen, sqrt

def critical_values_exact(f):
    x = f.parent().gen()
  
    fp = f.derivative()
    critical_points = fp.roots(multiplicities=False)
    critical_values = []
    
    for cp in critical_points:
        cv = f(cp)
        cv_radical = cv.radical_expression()
        critical_values.append((cv==cv_radical))

    return critical_values

def main():
    x = polygen(QQbar)
    t = 42
    v = t**4*(t-1)**3/(t+2)
    f = x**4 - 6*t**3*x**2 - 8*t**3*x
    g = -(t-1)**3*v/(3*(t+2)**3)*x^4 + 2*v*x**2 + 8/3*v*x - 8*t**4*(t**2+t+1)
    
    print("Polynomial f(x):") 
    print(f)
    print("Polynomial g(x):") 
    print(g)
    
    critical_values_f = critical_values_exact(f)
    critical_values_g = critical_values_exact(g)
    
    print("\nCritical values of f:")
    for i, cv in enumerate(critical_values_f, 1):
        print(f"Critical value {i}:\n {cv}")
        print()
        
    print("\nCritical values of g:")
    for i, cv in enumerate(critical_values_g, 1):
        print(f"Critical value {i}:\n {cv}")
        print()
        
if __name__ == "__main__":
    main()
\end{verbatim}in \textbf{SageMath} \cite{sagemath} returns the same critical values both in numerical and radical form:
\begin{verbatim}
Polynomial f(x): 
x^4 - 444528*x^2 - 592704*x
Polynomial g(x): 
-307935007631307/234256*x^4 + 107230600008/11*x^2 + 142974133344/11*x - 44982677376

Critical values of f:
Critical value 1: 
-4.912195494217215?e10 == -1/2*(482155633265041602566848512*I*sqrt(74087) +
4464098178074162551805159964672)^(1/3)*(I*sqrt(3) + 1) - 
135597252859697147904*(-I*sqrt(3) + 1)/(482155633265041602566848512*I*sqrt(74087) +
4464098178074162551805159964672)^(1/3) - 32934190464

Critical value 2: 
197568.1975316542? == (482155633265041602566848512*I*sqrt(74087) +
4464098178074162551805159964672)^(1/3) + 
271194505719394295808/(482155633265041602566848512*I*sqrt(74087) + 
4464098178074162551805159964672)^(1/3) - 32934190464

Critical value 3: 
-4.968081401802539?e10 == -1/2*(482155633265041602566848512*I*sqrt(74087) +
4464098178074162551805159964672)^(1/3)*(-I*sqrt(3) + 1) - 
135597252859697147904*(I*sqrt(3) + 1)/(482155633265041602566848512*I*sqrt(74087) +
4464098178074162551805159964672)^(1/3) - 32934190464

Critical values of g:
Critical value 1: 
-4.912195494217215?e10 == -1/2*(482155633265041602566848512*I*sqrt(74087) +
4464098178074162551805159964672)^(1/3)*(I*sqrt(3) + 1) - 
135597252859697147904*(-I*sqrt(3) + 1)/(482155633265041602566848512*I*sqrt(74087) +
4464098178074162551805159964672)^(1/3) - 32934190464

Critical value 2: 
-4.968081401802539?e10 == -1/2*(482155633265041602566848512*I*sqrt(74087) +
4464098178074162551805159964672)^(1/3)*(-I*sqrt(3) + 1) - 
135597252859697147904*(I*sqrt(3) + 1)/(482155633265041602566848512*I*sqrt(74087) +
4464098178074162551805159964672)^(1/3) - 32934190464

Critical value 3: 
197568.1975316542? == (482155633265041602566848512*I*sqrt(74087) +
4464098178074162551805159964672)^(1/3) + 
271194505719394295808/(482155633265041602566848512*I*sqrt(74087) +
4464098178074162551805159964672)^(1/3) - 32934190464
\end{verbatim}as expected.

\newpage
\bibliography{biblio}
\bibliographystyle{plain} 
\end{document}